\documentclass[10pt]{amsart}       
\usepackage{graphicx} 
\usepackage{ytableau}
\usepackage{amsaddr}
\usepackage{subcaption}
\usepackage[all]{xy}
\usepackage{float} 
\floatstyle{ruled}
\newfloat{Algorithm}{htbp}{algo}
\usepackage{pgf,tikz}
\usetikzlibrary{calc}
\usetikzlibrary{arrows, automata}
\usepackage{subcaption}
\usepackage{xcolor}
\usepackage{url}
\usepackage[utf8]{inputenc}
\usepackage[backend=biber, maxnames=4]{biblatex}
\usepackage{doi}
\usepackage{hyperref}    
\allowdisplaybreaks[1]

\DeclareFieldFormat{doi}{\href{https://dx.doi.org/#1}{\textsc{doi}}}
\DeclareFieldFormat{url}{\href{#1}{\textsc{url}}}
\renewbibmacro*{doi+eprint+url}{
  \printfield{doi}
  \newunit\newblock
  \iftoggle{bbx:eprint}{
    \usebibmacro{eprint}
  }{}
  \newunit\newblock
  \iffieldundef{doi}{
    \usebibmacro{url+urldate}}
  {}
}
\renewbibmacro{in:}{}
\AtEveryBibitem{
  \ifentrytype{book}
  {\clearfield{pages}}
}

\theoremstyle{plain} 
\newtheorem{lemma}{Lemma}
\newtheorem{corollary}[lemma]{Corollary}
\newtheorem{theorem}[lemma]{Theorem}
\newtheorem{proposition}[lemma]{Proposition}
\theoremstyle{remark}
\newtheorem{example}[lemma]{Example}
\newtheorem{remark}[lemma]{Remark}
\theoremstyle{definition}
\newtheorem{definition}[lemma]{Definition}

\numberwithin{lemma}{section}
\numberwithin{equation}{section}

\newcommand\Fq{\mathbf F_q}
\newcommand\Tab{\mathrm{Tab}}
\newcommand\T{\mathcal T}
\newcommand\ZZ{\mathbf Z}
\newcommand\arc{\mathrm{arc}}

\newcommand\nat{\mathbf P}
\newcommand\nn{\mathbf N}

\newcommand\stab{\mathrm{sTab}}
\newcommand\supp{\mathrm{supp}}

\renewcommand\AA{\mathcal{A}}
\renewcommand\subset{\subseteq}
\renewcommand\supset{\supseteq}  

\DeclareMathOperator\rank{rank}

\title[Set Partitions, Tableaux, and Profiles]{Set Partitions, Tableaux, and Subspace Profiles under Regular Diagonal Matrices}
\author{Amritanshu Prasad}
\address{The Institute of Mathematical Sciences, Chennai, India.}
\address{Homi Bhabha National Institute, Mumbai, India.}
\email{amri@imsc.res.in}

\author{Samrith Ram}
\address{Indraprastha Institute of Information Technology Delhi, New Delhi, India.}
\thanks{An extended abstract \cite{fpsac} describing some of these results appeared in the proceedings of FPSAC 2022.} 
\email{samrith@iiitd.ac.in}


 
\begin{document}
\begin{abstract}
  We introduce a family of univariate polynomials indexed by integer partitions. At prime powers, they count the number of subspaces in a finite vector space that transform under a regular diagonal matrix in a specified manner. This enumeration formula is a combinatorial solution to a problem introduced by Bender, Coley, Robbins and Rumsey. At $1$, they count set partitions with specified block sizes. At $0$, they count standard tableaux of specified shape. At $-1$, they count standard shifted tableaux of a specified shape. These polynomials are generated by a new statistic on set partitions (called the interlacing number) as well as a polynomial statistic on standard tableaux. They allow us to express $q$-Stirling numbers of the second kind as sums over standard tableaux and as sums over set partitions.  

  For partitions whose parts are at most two, these polynomials are the non-zero entries of the Catalan triangle associated to the $q$-Hermite orthogonal polynomial sequence.
  In particular, when all parts are equal to two, they coincide with the polynomials defined by Touchard that enumerate chord diagrams by the number of crossings.
  \end{abstract}

\maketitle

\tableofcontents
\section{Introduction}   
\label{sec:introduction}
In this paper, we define a family of polynomials $b_\lambda(q)$ indexed by integer partitions. The polynomial $b_\lambda(q)$ is a sum over $\Tab_{[n]}(\lambda)$, the set of standard tableaux of shape $\lambda$:
\begin{displaymath}
  b_\lambda(q) = \sum_{T\in \Tab_{[n]}(\lambda)} c_q(T),
\end{displaymath}
where $c_q(T)$ is a polynomial in $q$ associated to a standard tableau $T$.
The polynomial $c_q(T)$ arises out of a surprising new connection between two classical combinatorial classes, namely set partitions and standard tableaux and is defined as follows. For a standard tableau $T$ of shape $\lambda=(\lambda_1,\dotsc,\lambda_m)$, let $T_{ij}$ denote the entry in the $i$th row and $j$th column of $T$. For $1\leq i\leq m$ and $2\leq j\leq \lambda_i$ define
  \begin{displaymath}
    c_{ij}(T) = \#\{i'\mid i'\geq i \mbox{ and } T_{i',j-1}<T_{ij}\}.
  \end{displaymath}
See Example \ref{eg:ct} for a specific calculation. Then 
  \begin{equation}
    \label{eq:c}
    c_q(T) := \prod_{i=1}^m \prod_{j=2}^{\lambda_i} [c_{ij}(T)]_q,
  \end{equation}
  where $[n]_q=1+q+\cdots+q^{n-1}$ denotes the $q$-analog of the integer $n$.
  
In Section~\ref{sec:tabl-assoc-set} we associate a standard tableau to each partition of the set $[n]=\{1,\dotsc,n\}$.
Set partitions that map to a given tableau $T$ are counted by $c(T):=\prod_{i,j}c_{ij}(T)$.

The relationship between set partitions and tableaux gives rise to an algorithm (see Section~\ref{sec:enumeration}) for generating the set partitions corresponding to a given tableau.
By nesting this algorithm inside generators for standard tableaux, a new algorithm for the enumeration of set partitions with given block sizes is obtained.
This algorithm allows for more refined enumeration, for example, when the least elements of the blocks are specified.

In Section~\ref{sec:interl-stat}, we introduce a statistic on set partitions called the interlacing number, which is a variant of the well-known crossing number of a set partition (see, for example, \cite{MR2272140}). We show that the polynomial $c_q(T)$ is the generating polynomial of the interlacing statistic on the class of set partitions associated to the tableau $T$ (Theorem~\ref{theorem:statistic}). Moreover, noninterlacing set partitions of shape $\lambda$ are in bijection with standard tableaux of shape $\lambda$. Consequently, noninterlacing partitions of $[n]$ are in bijection with involutions in the symmetric group $S_n$ (see Corollary \ref{cor:noninterl}).

In Section~\ref{sec:minusone} we show that $b_\lambda(-1)$ counts the number of standard shifted tableaux of shape $\lambda$ for partitions whose parts are distinct, with the possible exception of the largest part.

When $q$ is a prime power, the polynomials $b_\lambda(q)$ can also be given a finite field interpretation. Let $\Fq$ be a finite field with $q$ elements and let $\Delta$ be a linear operator on $\Fq^n$. In this article, we focus on the case where $\Delta$ has $n$ distinct eigenvalues in $\Fq$.
Thus $\Delta$ is represented by a diagonal matrix with distinct entries on its diagonal, which we refer to as a regular diagonal matrix.

\begin{definition}
  \label{definition:profile}
  A subspace $W\subset \Fq^n$ has partial $\Delta$-profile $\mu=(\mu_1,\ldots,\mu_k)$ if 
  \begin{displaymath}
    \dim (W + \Delta W + \dotsb + \Delta^{j-1}W) = \mu_1+\dotsb+\mu_j \text{ for } 1\leq j\leq k.
  \end{displaymath}
Furthermore, if 
  \begin{displaymath}
    \dim (W + \Delta W + \dotsb + \Delta^{k-1}W) =   \dim (W + \Delta W + \dotsb + \Delta^kW),
  \end{displaymath}
  then we say that $W$ has $\Delta$-profile $\mu$.
\end{definition}
If $\mu=(\mu_1,\dotsc,\mu_k)$ is a partial profile of a subspace, then $\mu_j\geq \mu_{j+1}$ for all $j$, and so $\mu$ is a partition of some integer less than or equal to $n$ (see Lemma \ref{lem:ispartition}). The $\Delta$-profile of a subspace was introduced by Bender, Coley, Robbins and Rumsey \cite{MR1141317} who called it the dimension sequence. They considered the problem of determining the number of subspaces with $\Delta$-profile $\mu$ and gave explicit product formulas in the cases where $\Delta$ is regular nilpotent or has irreducible characteristic polynomial. 
Notions of cyclic vectors, splitting subspaces and anti-invariant subspaces can all be expressed in terms of profiles and partial profiles.

Let $\sigma_n(\mu)$ denote the number of subspaces with $\Delta$-profile $\mu$.
We prove that (see Theorem~\ref{theorem:counting-profile})
\begin{equation}
  \label{eq:sigmu}
 \sigma_n(\mu) = \binom n{|\mu|}(q-1)^{\sum_{j\geq 2}\mu_j}q^{\sum_{j\geq 2}\binom{\mu_j}2}b_{\mu'}(q),
\end{equation}
where $\mu'$ denotes the partition conjugate to $\mu$.

To summarize, the polynomials $b_\lambda(q)$ have the following specializations:
\begin{enumerate}
\item When $q=1$, they count partitions of $[n]$ with block sizes given by the parts of $\lambda$ (Eq. \eqref{eq:pinlam}). 
\item When $q=0$, they count standard tableaux of shape $\lambda$ and also the number of noninterlacing set partitions of shape $\lambda$ (Corollary \eqref{cor:noninterl} and Eq. \eqref{eq:blamset}). 
\item When $q=-1$ and the parts of $\lambda$ are distinct, with the possible exception of the largest part, they count standard shifted tableaux of shape $\lambda$ (Theorem \ref{theorem:minusone}).
 \item When $q$ is a prime power, they count subspaces of $\Fq^n$ with $\Delta$-profile $\lambda'$, up to a factor of the form $q^a(q-1)^b$ (Eq. \eqref{eq:sigmu}). 
\end{enumerate}
We describe a recursive formula for $b_\lambda(q)$ in Section~\ref{sec:rec-comp}, which allows for faster computation.
When $\lambda$ has all parts of size at most two, the polynomials $b_\lambda(q)$ are the non-zero entries of the Catalan triangle associated to the $q$-Hermite orthogonal polynomial sequence (see Section~\ref{sec:relation-q-hermite}).
In particular, when $\lambda=(2^m)$, $b_\lambda(q)=T_m(q)$, the polynomial defined by Touchard that enumerates chord diagrams for $2m$ points arranged in a circle by number of crossings. When $\lambda$ has at most two parts, we give a bijective proof of a closed formula for $b_{\lambda}(q)$ in Section~\ref{sec:twoparts}.  

We express the $q$-Stirling numbers of the second kind in terms of the polynomials $b_\lambda(q)$, as sums over standard tableaux, and as sums over set partitions:
\begin{align}
  S_q(n,m) &= \sum_{\lambda\vdash n,\;l(\lambda)=m} q^{\sum_i(i-1)(\lambda_i-1)}b_\lambda(q), \label{eq:qs1}\\
    S_q(n,m)&= \sum_{\lambda \vdash n,\;l(\lambda)=m}q^{\sum_i (i-1)(\lambda_i-1)}\sum_{T\in \Tab_{[n]}(\lambda)} c_q(T), \label{eq:qs2}\\
  S_q(n,m) &= \sum_{\AA\in \Pi_{n,m}} q^{v(\AA)+\sum_i(i-1)(\lambda_i^\AA-1)}.    \label{eq:qs3}
\end{align} 

Here $\Pi_{n,m}$ denotes the collection of partitions of $[n]$ having $m$ blocks, $(\lambda_1^\AA,\lambda_2^\AA,\ldots)$ is the list of block sizes of $\AA$ sorted in weakly decreasing order, and $v(\AA)$ is the interlacing number of $\AA$.

Many statistics on set partitions are known to produce the $q$-Stirling numbers of the second kind  \cite{CAI201750,MR834272,MR644671,MR1082841}. Our statistic appears to be different from all of these. It is worth noting that Milne \cite{MR511401} relates $q$-Stirling numbers of the second kind to a counting problem over finite fields. However,  that problem appears to be unrelated to the counting problem on finite fields considered in this paper. Cai and Readdy \cite[Thm. ~3.2]{CAI201750} express $S_q(n,m)$ as a sum of expressions of the form $q^a(q+1)^b$ over a small class of set partitions. In this vein, we also express $S_q(n,m)$ as a sum over noninterlacing set partitions (see Eq. \eqref{eq:qs4}) but in our case, the summands are powers of $q$ times a product of $q$-integers.

Combining Eqs. \eqref{eq:sigmu}  and \eqref{eq:qs1} allows for a combinatorial interpretation of $q$-Stirling numbers of the second kind (Lemma~\ref{lemma:combin-q-stir}), which in turn leads to a bijective interpretation of an identity of Carlitz expressing $q$-binomial coefficients in terms of $q$-Stirling numbers of the second kind (see Eq. \eqref{eqn:carlitz}).

At $q=1$, the identity \eqref{eq:qs2} gives an expression for the Stirling numbers of the second kind as the sum of a statistic on standard tableaux which appears to be new. A similar identity holds for the Bell numbers (see Section \ref{sec:setstotableaux}).

If $m$ divides $n$, say $n=md$, then a $\Delta$-splitting subspace of $\Fq^n$ is a subspace $W$ of dimension $m$ such that
\begin{equation}
  \label{eq:defsplit}
  W\oplus \Delta W\oplus \cdots\oplus \Delta^{d-1}W=\Fq^n.  
\end{equation}
Thus an $m$-dimensional splitting subspace is a subspace of $\Fq^n$ with $\Delta$-profile $(m^d)$. The definition of a splitting subspace can be traced back to the work of Niederreiter \cite{N2} on pseudorandom number generation. Niederreiter, apparently unaware of the work of Bender, Coley, Robbins and Rumsey, asked for a formula for the number of $\Delta$- splitting subspaces when $\Delta$ has irreducible characteristic polynomial. The case where the invariant factors of $\Delta$ satisfy certain degree constraints was resolved in \cite{polynomialmatrices}. In the case where $\Delta$ has $n$ distinct eigenvalues in $\Fq$, Eq. \eqref{eq:sigmu} gives the number of splitting subspaces of dimension $m$ when we set $\mu=(m^d)$; see Section~\ref{sec:chord}. 

In \cite{touchard}, the number of splitting subspaces of dimension $m$ for each operator $\Delta$ on $\Fq^{2m}$ is shown to be
\begin{displaymath}
  q^{\binom m2}\sum_{j=0}^{2m}(-1)^j X^\Delta_jq^{\binom{m-j+1}2},
\end{displaymath}
where $X^\Delta_j$ is the number of $j$-dimensional $\Delta$-invariant subspaces.
In the case where $\Delta$ is diagonal with distinct diagonal entries, combining this formula with our results gives a new proof of the Touchard-Riordan formula \eqref{eq:touchard-riordan}.

In Section \ref{sec:partialprofiles} we extend our results to the enumeration of subspaces with a given partial profile.
An interesting special case is that of anti-invariant subspaces. A subspace $W$ of $\Fq^n$ is said to be $l$-fold $\Delta$-anti-invariant if
\begin{equation}
  \label{eq:defanti}
  \dim (W+\Delta W+\cdots+\Delta^{l}W)=(l+1)\cdot \dim W.
\end{equation}
 
Anti-invariant subspaces were studied by Barr{\'i}a and Halmos \cite{MR748946} who determined their maximum possible dimension in the case $l=1$. This result was extended by Kn{\"u}ppel and Nielsen \cite{MR2013452} to arbitrary $l$. It is clear from the definition that an $m$-dimensional subspace $W$ is $l$-fold $\Delta$-anti-invariant if and only if it has partial $\Delta$-profile $\mu=(m^{l+1})$. In the case where $\Delta$ has $n$ distinct eigenvalues in $\Fq$, Corollary~\ref{cor:numantiinv} gives a formula for the number of $\Delta$-anti-invariant subspaces.

The general case of the problem posed by Bender, Coley, Robbins and Rumsey on determining the number of subspaces with a given profile has recently been settled \cite{ram2023} and depends on results in \cite{ramschlosser} which in turn use ideas from this paper.


\ytableausetup{smalltableaux}
\section{Set partitions} 
\label{sec:set-partitions}
\subsection{Set partitions and standard notation}
Let $S$ be a finite subset of the set $\nat$ of positive integers.
A partition $\AA = \{A_1,\dotsc,A_m\}$ of $S$ is a decomposition
\begin{displaymath}
  S = A_1\cup \dotsb \cup A_m,
\end{displaymath}
where $A_1,\dotsc,A_m$ are pairwise disjoint non-empty subsets of $S$.
The subsets\linebreak $A_1,\ldots,A_m$ are called the \emph{blocks} of $\AA$.
The order of the blocks does not matter.
Following standard conventions \cite[\S~2.7.1.5]{TACOP4A}, the elements of each block are listed in increasing order, and the blocks are listed in increasing order of their least elements.
When this is the case, we write $\AA=A_1|\dotsb|A_m$, which we call \emph{the standard notation} for $\AA$.
The shape of a set partition is the list of cardinalities of $A_1,\dotsc,A_m$, sorted in weakly decreasing order.
Thus the shape of a partition of $S$ is an integer partition of $|S|$.
Denote the set of all partitions of $S$ by $\Pi_S$, and the set of all partitions of $S$ with shape $\lambda$ by $\Pi_S(\lambda)$. 
\begin{example}
  The standard notation for the partition
  \begin{displaymath}
    \AA=(\{6,5\},\{8,9,3\},\{1,2\})
  \end{displaymath}
  of $S=\{1,2,3,5,6,8,9\}$ is $\AA = 12|389|56$.
  The shape of $\AA$ is $(3,2,2)$.
\end{example} 
\subsection{The tableau associated to a set partition}
\label{sec:tabl-assoc-set}
A tableau is a finite array of the form
\begin{displaymath}
  T = (T_{ij}\mid T_{ij}\in \nat,\;1\leq i\leq m,\; 1\leq j\leq \lambda_i),
\end{displaymath}
where $\lambda=(\lambda_1,\dotsc,\lambda_m)$ is an integer partition.
The tableaux $T$ is displayed with the number $T_{ij}$ filled into the cell in the $i$th row and $j$th column of an array of cells.
The partition $\lambda$ is called the shape of $T$.
\begin{definition}
  \label{defn:T}
  Given $\AA=A_1|\dotsb|A_m\in \Pi_S$, form an array whose entry in the $i$th row and $j$th column is the $j$th smallest element of $A_i$.
  Then sort and top-justify the columns of this array.
  Denote the resulting tableau by $\T(\AA)$.
\end{definition}
\begin{example}\label{example:partition-tableau}
  When $\AA = 12|389|56$, the associated array is
  \begin{displaymath}
    \ytableaushort{12\none,389,5 6 \none}.
  \end{displaymath}
  Sorting and top-justifying the columns gives the tableau
  \begin{displaymath}
    \T(\AA) = \ytableaushort{129,36,58}.
  \end{displaymath}
\end{example}
\begin{definition}
  A tableau $T$ is called a \emph{multilinear tableau} if each element of $\nat$ occurs at most once in $T$, the rows of $T$ increase from left to right, and the columns of $T$ increase from top to bottom.

  The set of integers that occur in $T$ is called the support of $T$ and is denoted $\supp(T)$. For each $S\subset \nat$, let $\Tab_S$ denote the set of multilinear tableaux with support $S$. For each integer partition $\lambda$, denote the set of multilinear tableaux of shape $\lambda$ and support $S$ by $\Tab_S(\lambda)$. Let $\Tab_{\subset S}(\lambda)$ denote the set of multilinear tableaux of shape $\lambda$ whose support is a subset of $S$. 
  The usual notion of a standard Young tableau coincides with that of a multilinear tableau with support $[n]$ for some $n\geq 0$.
\end{definition}
\begin{example}
  The tableau in Example~\ref{example:partition-tableau} is a multilinear tableau of shape $(3,2,2)$ and support $\{1,2,3,5,6,8,9\}$.
\end{example}
\begin{lemma}
  Let $\lambda$ be an integer partition and $S\subset \nat$ be finite.
  For every $\AA\in \Pi_S(\lambda)$, $\T(\AA)$ is a multilinear tableau of shape $\lambda$ and support $S$.
\end{lemma}
\begin{proof}
  The number of elements in the $j$th column of $T(\AA)$ is clearly the number of blocks of $\AA$ whose length is greater than or equal to $j$.
  Therefore the shape of $\T(\AA)$ is $\lambda$.
  The columns of $\T(\AA)$ are increasing by construction, and its entries are precisely the elements of $S$.
  To see that the rows are increasing, proceed by induction on the cardinality of $S$.
  When $|S|=1$, the result is clear.
  Suppose that $|S|>1$.
  Let $n$ be the largest element of $S$ and let $S'=S-\{n\}$.
  Each partition $\AA$ of $S$ is obtained by adding $n$ to one of the blocks of a partition $\tilde\AA$ of $S'$.
  Therefore $\T(\AA)$ is obtained from $\T(\tilde\AA)$ by appending $n$ to the end of one of its columns.
  Since $n$ is larger than every entry of $\T(\tilde\AA)$ it must be greater than all the elements in its row.
  Therefore $\T(\AA)$ also has increasing rows.
\end{proof}
Definition~\ref{defn:T} gives rise to a function
\begin{displaymath}
  \T:\Pi_S(\lambda) \to \Tab_S(\lambda).
\end{displaymath}
It is easy to see that $\T$ maps $\Pi_S(\lambda)$ surjectively onto $\Tab_S(\lambda)$: given $T\in \Tab_S(\lambda)$ simply take $\AA$ to be the partition of $S$ whose blocks are the rows of $T$.
Then $\T(\AA)=T$.
For each $T\in \Tab_S(\lambda)$, let
\begin{displaymath}
  \Pi(T) = \{\AA\in \Pi_S(\lambda) \mid \T(\AA)=T\},
\end{displaymath}
the fibre of $\T$ over $T$.
Our next goal is to understand the set $\Pi(T)$ for each multilinear tableau $T$.
\begin{example}
  \label{example:tableau}
  When $T=\ytableaushort{129,36,58}$, the set partitions in $\Pi(T)$ are
  \begin{displaymath}
    129|38|56,\; 129|36|58,\; 12|369|58,\; 12|38|569,\; 12|36|589,\; 12|389|56.
  \end{displaymath}
\end{example}
\subsection{Set partitions associated to a given tableau}
Suppose that $S\subset \nat$ is finite, and $\lambda$ is an integer partition of $|S|$.

\begin{definition}
  Let $T$ be a multilinear tableau of shape $\lambda=(\lambda_1,\dotsc,\lambda_m)$.
  For $1\leq i\leq m$ and $2\leq j\leq \lambda_i$ define
  \begin{displaymath}
    c_{ij}(T) = \#\{i'\mid i'\geq i \mbox{ and } T_{i',j-1}<T_{ij}\}.
  \end{displaymath}
  Define
  \begin{displaymath}
    c(T) = \prod_{i=1}^m \prod_{j=2}^{\lambda_i} c_{ij}(T).
  \end{displaymath}
\end{definition}
\begin{example}\label{eg:ct}
  For the tableau $T$ of Example~\ref{example:tableau}, the values of $c_{ij}(T)$ for each relevant cell are shown below:

  \begin{displaymath}
    \ytableaushort{{}13,{}2,{}1},
  \end{displaymath}
  and $c(T)=6$.
\end{example}
\begin{theorem}
  \label{theorem:fibres}
  For every multilinear tableau $T$, the cardinality of $\Pi(T)$ is $c(T)$.
\end{theorem}
\begin{proof}
  Induct on $k$, the number of columns of $T$.
  If $k=1$, then $c(T)$ is the empty product, equal to $1$.
  There is only one set partition with $\T(\AA)=T$, namely the partition with singleton blocks.
  Therefore the theorem holds for $k=1$.

  Let $C$ denote the set of entries of the first column of $T$, let $\tilde T$ denote the tableau obtained by deleting the first column of $T$, and let $\tilde C$ denote the set of entries of the first column of $\tilde T$ (the same as the second column of $T$).
  Given $\AA$ such that $\T(\AA)=T$, let $\tilde\AA$ denote the partition obtained from $\AA$ by removing the least element of each block (and discarding the empty blocks that result).
  Then $\T(\tilde\AA)=\tilde T$.
  Thus $\AA\mapsto \tilde\AA$ gives rise to a function $\phi:\Pi(T)\to \Pi(\tilde T)$.

  Let
  \begin{displaymath}
    M(C,\tilde C):=\{\alpha: \tilde C\to C \mid \alpha \text{ injective with }\alpha(\tilde c)<\tilde c \text{ for each }\tilde c\in \tilde C\}.
  \end{displaymath}
  Given $\AA\in \phi^{-1}(\tilde\AA)$, define $\alpha\in M(C,\tilde C)$ by setting $\alpha(\tilde c)=c$, where $c$ is the least element in the block of $\AA$ containing $\tilde c$.
  Define $\Phi:\Pi(T)\to M(C,\tilde C)\times \Pi(\tilde T)$ by $\Phi(\AA)=(\alpha,\phi(\AA))$.
  This construction gives rise to a bijection
  \begin{displaymath}
    \Phi:\Pi(T)\to M(C,\tilde C)\times \Pi(\tilde T).
  \end{displaymath}
  Therefore $\phi^{-1}(\tilde\AA)$ is in bijective correspondence with $M(C,\tilde C)$.

  Suppose that $\tilde C=(\tilde c_1,\tilde c_2,\dotsc,\tilde c_m)$ in increasing order.
  For $\alpha(\tilde c_1)$, there are $c_{12}(T)$ choices.
  Having made that choice, there are precisely $c_{22}(T)$ choices for $\alpha(\tilde c_2)$.
  Proceeding in this manner, we see that there are $\prod_{\{i\mid \lambda_i\geq 2\}}c_{i2}(T)$ possibilities for $\alpha$.
  It follows that
  \begin{displaymath}
    |\Pi(T)| = |\Pi(\tilde T)| \prod_{\{i\mid \lambda_i\geq 2\}}c_{i2}(T) = c(\tilde T) \prod_{\{i\mid \lambda_i\geq 2\}}c_{i2}(T)=c(T),
  \end{displaymath}
  by induction.
\end{proof}
\subsection{Counting set partitions using tableaux}
\label{sec:setstotableaux}
Recall that, for every partition $\lambda$ of $n$, $\Pi_n(\lambda)$ denotes the set of partitions of $[n]$ of shape $\lambda$. 
\begin{corollary}
  \label{corollary:pinla}
Let $n$ be a nonnegative integer.
  \begin{enumerate}
\item  For each integer partition $\lambda$ of $n$,
  \begin{equation}
    \label{eq:pinlam}
    |\Pi_{n}(\lambda)| = \sum_{T\in \Tab_{[n]}(\lambda)}c(T).
  \end{equation}
\item For each nonnegative integer $m$, the Stirling number of the second kind is given by
$$
    S(n,m)=\sum_{\substack{\lambda\vdash n\\l(\lambda)=m}}\sum_{T\in \Tab_{[n]}(\lambda)}c(T).
$$
\item   For every positive integer $n$, the number of partitions of $[n]$ is given by
  \begin{displaymath}
    B_n = \sum_{T\in \Tab_{[n]}} c(T),
  \end{displaymath}
  where $\Tab_{[n]}$ denotes the set of standard tableaux with $n$ cells.
  \end{enumerate}
\end{corollary}

 \subsection{Generation of set partitions}
\label{sec:enumeration}
The proof of Theorem~\ref{theorem:fibres} suggests the use of Algorithm~\ref{algorithm:generator} for the generation of all set partitions in $\Pi(T)$ for every multilinear tableau $T$.
By nesting Algorithm~\ref{algorithm:generator} inside a generator for $\Tab_S(\lambda)$, a generator for all set partitions with shape $\lambda$ is obtained.
\begin{Algorithm}
  \caption{Algorithm to generate elements of $\Pi(T)$}
  \label{algorithm:generator}
  \begin{itemize}
  \item \textbf{If} $T$ has only one column
    \begin{itemize}
    \item \textbf{Visit} the partition of the support of $T$ with singleton blocks and stop.
    \end{itemize}
  \item \textbf{Else}
    \begin{itemize}
    \item Let $\tilde T$ be the tableau obtained by removing the first column of $T$, $C$ and $\tilde C$ be the first columns of $T$ and $\tilde T$ respectively.
    \item \textbf{For }each injective map $\alpha: \tilde C\to C$ with $\tilde c<\alpha(\tilde c)$ and $\tilde\AA\in \Pi(\tilde T)$,
      \begin{itemize}
      \item \textbf{Visit} the partition $\AA$ obtained by adding $\alpha(\tilde c)$ to the block of $\tilde\AA$ containing $\tilde c$ for each $\tilde c\in \tilde C$.
      \end{itemize}
    \end{itemize}
  \end{itemize}
\end{Algorithm} 

\section{$q$-analogs}
\subsection{A $q$-analog of $c(T)$}
For each multilinear tableau $T$ of shape $\lambda=(\lambda_1,\dotsc,\lambda_m)$, define $c_q(T)\in \ZZ[q]$ by
\begin{equation}
  \label{eq:rq}
  c_q(T) = \prod_{i=1}^m \prod_{j=2}^{\lambda_i} [c_{ij}(T)]_q,
\end{equation}
where, for each positive integer $n$, $[n]_q=1+q+\dotsb+q^{n-1}$, the $q$-analog of $n$.
Clearly $c_q(T)$ is a polynomial with nonnegative integer coefficients, and substituting $q=1$ gives
\begin{displaymath}
  c_q(T)|_{q=1} = c(T),
\end{displaymath}
the cardinality of $\Pi(T)$.
\subsection{The polynomials $b_\lambda(q)$}
\label{sec:polyn-b_lambd}
For each integer partition $\lambda$ and each positive integer $n$, define
\begin{equation}
  \label{eq:bla}
  b_\lambda^n(q) = \sum_{T\in \Tab_{\subset [n]}(\lambda)} c_q(T).
\end{equation}
Write $b_\lambda(q)$ for $b_\lambda^n(q)$ when $\lambda$ is a partition of $n$. Every order preserving relabeling of the entries of $T$ leaves $c_q(T)$ invariant. Therefore,
$$
b_\lambda^n(q)={n \choose |\lambda|}b_\lambda(q).
$$
By Corollary~\ref{corollary:pinla}, we have
\begin{displaymath}
  b_\lambda(1) = |\Pi_{n}(\lambda)|,
\end{displaymath}
and since, for every standard tableau $T,$ $c_q(T)$ is a product of $q$-integers, we have
\begin{displaymath}
  b_\lambda(0) = |\Tab_{[n]}(\lambda)|.
\end{displaymath}
\subsection{The interlacing statistic}  
\label{sec:interl-stat}
Let $\nat^*$ denote the ordered set $\nat \cup \{\infty\}$, where $\infty$ is deemed to be greater than every element of $\nat$.
\begin{definition}[Crossing arcs]
  For $a,b,c,d\in \nat^*$, we say that the arcs $(a,b)$ and $(c,d)$ cross if the intervals $[a,b]$ and $[c,d]$ are neither nested, nor disjoint.
  In other words:
  \begin{gather*}
    \text{either } a<c<b<d, \text{ or } c<a<d<b.
  \end{gather*}
\end{definition}
\begin{figure}[h]
  \begin{center}
    \begin{minipage}{\textwidth}
      \centering
      \begin{tikzpicture}
        [scale=1,every node/.style={circle,fill=black,inner sep=2pt, minimum size=6pt}]
        \node[label=below:$a$] (1) at (1,0) {};
        \node[label=below:$c$] (2) at (2,0) {};
        \node[label=below:$d$] (3) at (3,0) {};
        \node[label=below:$b$] (4) at (4,0) {};
        \draw[thick]
        (1) [out=45, in=135] to  (4);
        \draw[thick]
        (2) [out=45, in=135] to  (3);
      \end{tikzpicture}
      \quad
      \centering
      \begin{tikzpicture}
        [scale=1,every node/.style={circle,fill=black,inner sep=2pt, minimum size=6pt}]
        \node[label=below:$a$] (1) at (1,0) {};
        \node[label=below:$c$] (2) at (2,0) {};
        \node[label=below:$b$] (3) at (3,0) {};
        \node[label=below:$d$] (4) at (4,0) {};
        \draw[thick]
        (1) [out=45, in=135] to  (3);
        \draw[thick]
        (2) [out=45, in=135] to  (4);
      \end{tikzpicture}\\
      \centering
      \begin{tikzpicture}
        [scale=1,every node/.style={circle,fill=black,inner sep=2pt, minimum size=6pt}]
        \node[label=below:$a$] (1) at (1,0) {};
        \node[label=below:$b$] (2) at (2,0) {};
        \node[label=below:$c$] (3) at (3,0) {};
        \node[label=below:$d$] (4) at (4,0) {};
        \draw[thick]
        (1) [out=45, in=135] to  (2);
        \draw[thick]
        (3) [out=45, in=135] to  (4);
      \end{tikzpicture}
      \quad
      \centering
      \begin{tikzpicture}
        [scale=1,every node/.style={circle,fill=black,inner sep=2pt, minimum size=6pt}]
        \node[label=below:$c$] (1) at (1,0) {};
        \node[label=below:$a$] (2) at (2,0) {};
        \node[label=below:$d$] (3) at (3,0) {};
        \node[label=below:$b$] (4) at (4,0) {};
        \draw[thick]
        (1) [out=45, in=135] to  (3);
        \draw[thick]
        (2) [out=45, in=135] to  (4);
      \end{tikzpicture}
      \caption{Left: noncrossing arcs, Right: crossing arcs.}
      \label{fig:interlacing-non-interlacing}
    \end{minipage}
  \end{center}
\end{figure}
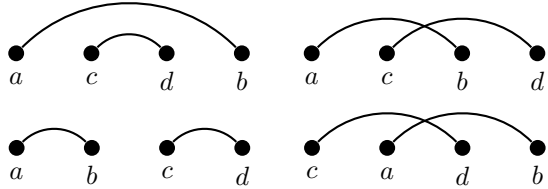
In Figure~\ref{fig:interlacing-non-interlacing}, arcs $(a,b)$ and $(c,d)$ do not cross in the diagrams on the left, and do cross in the diagrams on the right.
\begin{definition}
  [The arcs of a set]
  Given a set $A\subset \nat$ whose elements are $a_1,\dotsc,a_l$ in increasing order, its $j$th arc is the pair $\arc_j(A) = (a_j,a_{j+1})$ for $j=1,\dotsc,l-1$, and its $l$th arc is $\arc_l(A) = (a_l,\infty)$.
\end{definition}
\begin{definition}[Interlacing]
  \label{definition:interlacing}
  Let $S$ be any finite subset of $\nat$.
  Let $\AA = A_1|\dotsb|A_m\in \Pi_S$ with $|A_i|=l_i$.
  An \emph{interlacing} of $\AA$ is a pair $(\arc_j(A_i)$, $\arc_j(A_{i'}))$ of crossing arcs  for some $1\leq i<i'\leq m$ and some $1\leq j\leq \min(l_i,l_{i'})$.
  Let $v(\AA)$ denote the total number of interlacings of the set partition $\AA$, called the \emph{interlacing number} of $\AA$.
\end{definition}
\begin{remark}
  The notion of interlacing for a set partition appears to be similar to the notion of \emph{crossing} defined by Kreweras \cite{Kreweras}.
  For set partitions where all blocks are of size two, interlacings coincide with crossings.
  However, in general, a set partition can be noninterlacing without being noncrossing, and vice versa.
  For example $13|2$ is noncrossing, but has interlacing arcs $(1,3)$ and $(2,\infty)$.
  On the other hand, $124|35$ has interlacing number zero but is not noncrossing.
  The arcs $(2,4)$ and $(3,5)$ do cross, but $(2,4)$ is a second arc whereas $(3,5)$ is a first arc.
\end{remark} 
\begin{table}
  \begin{center}
    \begin{tabular}{ccc}
      $\AA$ & Arcs of $\AA$ & $v(\AA)$\\
      \hline
      $129|38|56$ & 
                    \begin{minipage}{0.75\textwidth}
                      \begin{tikzpicture}
                        [scale=0.7,every node/.style={circle,fill=black,inner sep=2pt, minimum size=6pt}]
                        \node[label=below:$1$] (1) at (1,0) {};
                        \node[label=below:$2$] (2) at (2,0) {};
                        \node[label=below:$3$] (3) at (3,0) {};
                        \node[label=below:$5$] (5) at (5,0) {};
                        \node[label=below:$6$] (6) at (6,0) {};
                        \node[label=below:$8$] (8) at (8,0) {};
                        \node[label=below:$9$] (9) at (9,0) {};
                        \node[label=below:$\infty$] (0) at (10,0) {};
                        \draw[thick,color=teal]
                        (1) [out=45, in=135] to  (2);
                        \draw[thick,color=orange]
                        (2) [out=45, in=135] to  (9);
                        \draw[thick,color=teal]
                        (3) [out=45, in=135] to  (8);
                        \draw[thick,color=teal]
                        (5) [out=45, in=135] to  (6);
                        \draw[thick,color=orange]
                        (6) [out=45, in=135] to  (0);
                        \draw[thick,color=orange]
                        (8) [out=45, in=135] to  (0);
                        \draw[thick,color=blue]
                        (9) [out=45, in=135] to  (0);
                      \end{tikzpicture}
                    \end{minipage}
                            &  2\\
      $129|36|58$&
                   \begin{minipage}{0.75\textwidth}
                     \begin{tikzpicture}
                       [scale=0.7,every node/.style={circle,fill=black,inner sep=2pt, minimum size=6pt}]
                       \node[label=below:$1$] (1) at (1,0) {};
                       \node[label=below:$2$] (2) at (2,0) {};
                       \node[label=below:$3$] (3) at (3,0) {};
                       \node[label=below:$5$] (5) at (5,0) {};
                       \node[label=below:$6$] (6) at (6,0) {};
                       \node[label=below:$8$] (8) at (8,0) {};
                       \node[label=below:$9$] (9) at (9,0) {};
                       \node[label=below:$\infty$] (0) at (10,0) {};
                       \draw[thick,color=teal]
                       (1) [out=45, in=135] to  (2);
                       \draw[thick,color=orange]
                       (2) [out=45, in=135] to  (9);
                       \draw[thick,color=teal]
                       (3) [out=45, in=135] to  (6);
                       \draw[thick,color=teal]
                       (5) [out=45, in=135] to  (8);
                       \draw[thick,color=orange]
                       (6) [out=45, in=135] to  (0);
                       \draw[thick,color=orange]
                       (8) [out=45, in=135] to  (0);
                       \draw[thick,color=blue]
                       (9) [out=45, in=135] to  (0);
                     \end{tikzpicture}
                   \end{minipage}
                            & $3$\\
      $12|369|58$ &
                    \begin{minipage}{0.75\textwidth}
                      \begin{tikzpicture}
                        [scale=0.7,every node/.style={circle,fill=black,inner sep=2pt, minimum size=6pt}]
                        \node[label=below:$1$] (1) at (1,0) {};
                        \node[label=below:$2$] (2) at (2,0) {};
                        \node[label=below:$3$] (3) at (3,0) {};
                        \node[label=below:$5$] (5) at (5,0) {};
                        \node[label=below:$6$] (6) at (6,0) {};
                        \node[label=below:$8$] (8) at (8,0) {};
                        \node[label=below:$9$] (9) at (9,0) {};
                        \node[label=below:$\infty$] (0) at (10,0) {};
                        \draw[thick,color=teal]
                        (1) [out=45, in=135] to  (2);
                        \draw[thick,color=teal]
                        (3) [out=45, in=135] to  (6);
                        \draw[thick,color=orange]
                        (6) [out=45, in=135] to  (9);
                        \draw[thick,color=teal]
                        (5) [out=45, in=135] to  (8);
                        \draw[thick,color=orange]
                        (2) [out=45, in=135] to  (0);
                        \draw[thick,color=orange]
                        (8) [out=45, in=135] to  (0);
                        \draw[thick,color=blue]
                        (9) [out=45, in=135] to  (0);
                      \end{tikzpicture}
                    \end{minipage}
                            & $2$\\
      $12|38|569$&
                   \begin{minipage}{0.75\textwidth}
                     \begin{tikzpicture}
                       [scale=0.7,every node/.style={circle,fill=black,inner sep=2pt, minimum size=6pt}]
                       \node[label=below:$1$] (1) at (1,0) {};
                       \node[label=below:$2$] (2) at (2,0) {};
                       \node[label=below:$3$] (3) at (3,0) {};
                       \node[label=below:$5$] (5) at (5,0) {};
                       \node[label=below:$6$] (6) at (6,0) {};
                       \node[label=below:$8$] (8) at (8,0) {};
                       \node[label=below:$9$] (9) at (9,0) {};
                       \node[label=below:$\infty$] (0) at (10,0) {};
                       \draw[thick,color=teal]
                       (1) [out=45, in=135] to  (2);
                       \draw[thick,color=teal]
                       (5) [out=45, in=135] to  (6);
                       \draw[thick,color=orange]
                       (6) [out=45, in=135] to  (9);
                       \draw[thick,color=teal]
                       (3) [out=45, in=135] to  (8);
                       \draw[thick,color=orange]
                       (2) [out=45, in=135] to  (0);
                       \draw[thick,color=orange]
                       (8) [out=45, in=135] to  (0);
                       \draw[thick,color=blue]
                       (9) [out=45, in=135] to  (0);
                     \end{tikzpicture}
                   \end{minipage}
                            & $1$\\
      $12|36|589$ &
                    \begin{minipage}{0.75\textwidth}
                      \begin{tikzpicture}
                        [scale=0.7,every node/.style={circle,fill=black,inner sep=2pt, minimum size=6pt}]
                        \node[label=below:$1$] (1) at (1,0) {};
                        \node[label=below:$2$] (2) at (2,0) {};
                        \node[label=below:$3$] (3) at (3,0) {};
                        \node[label=below:$5$] (5) at (5,0) {};
                        \node[label=below:$6$] (6) at (6,0) {};
                        \node[label=below:$8$] (8) at (8,0) {};
                        \node[label=below:$9$] (9) at (9,0) {};
                        \node[label=below:$\infty$] (0) at (10,0) {};
                        \draw[thick,color=teal]
                        (1) [out=45, in=135] to  (2);
                        \draw[thick,color=teal]
                        (3) [out=45, in=135] to  (6);
                        \draw[thick,color=orange]
                        (8) [out=45, in=135] to  (9);
                        \draw[thick,color=teal]
                        (5) [out=45, in=135] to  (8);
                        \draw[thick,color=orange]
                        (2) [out=45, in=135] to  (0);
                        \draw[thick,color=orange]
                        (6) [out=45, in=135] to  (0);
                        \draw[thick,color=blue]
                        (9) [out=45, in=135] to  (0);
                      \end{tikzpicture}
                    \end{minipage}
                            &$1$\\
      $12|389|56$ &
                    \begin{minipage}{0.75\textwidth}
                      \begin{tikzpicture}
                        [scale=0.7,every node/.style={circle,fill=black,inner sep=2pt, minimum size=6pt}]
                        \node[label=below:$1$] (1) at (1,0) {};
                        \node[label=below:$2$] (2) at (2,0) {};
                        \node[label=below:$3$] (3) at (3,0) {};
                        \node[label=below:$5$] (5) at (5,0) {};
                        \node[label=below:$6$] (6) at (6,0) {};
                        \node[label=below:$8$] (8) at (8,0) {};
                        \node[label=below:$9$] (9) at (9,0) {};
                        \node[label=below:$\infty$] (0) at (10,0) {};
                        \draw[thick,color=teal]
                        (1) [out=45, in=135] to  (2);
                        \draw[thick,color=teal]
                        (5) [out=45, in=135] to  (6);
                        \draw[thick,color=orange]
                        (8) [out=45, in=135] to  (9);
                        \draw[thick,color=teal]
                        (3) [out=45, in=135] to  (8);
                        \draw[thick,color=orange]
                        (2) [out=45, in=135] to  (0);
                        \draw[thick,color=orange]
                        (6) [out=45, in=135] to  (0);
                        \draw[thick,color=blue]
                        (9) [out=45, in=135] to  (0);
                      \end{tikzpicture}
                    \end{minipage}
                            & $0$\\
      \hline
    \end{tabular}
  \end{center}

  \caption{Statistics for set partitions corresponding to $\ytableaushort{129,36,58}$}
  \label{tab:statistics}
\end{table}
Table~\ref{tab:statistics} shows the arcs and the number of interlacings for the set partitions in Example~\ref{example:tableau}.
The first, second, and third arcs are shown in different colours.
Only crossing arcs of the same colour contribute to the interlacing number.

The following theorem gives an enumerative interpretation of the coefficients of $c_q(T)$.
\begin{theorem}
  \label{theorem:statistic}
  For every multilinear tableau $T$,
  \begin{equation}
    \label{eq:rq-sp}
    c_q(T) = \sum_{\T(\AA)=T} q^{v(\AA)}.
  \end{equation}
\end{theorem}
\begin{proof}
  Induct on $k$, the number of columns of $T$.
  When $T$ has one column, $\Pi(T)$ has only one element, namely the partition of $S$ with singleton blocks.
  This partition has no interlacings, consistent with $c_q(T)=1$.

  Now suppose $k>1$.
  Recall the bijection $\Phi:\Pi(T)\to M(C,\tilde C) \times\Pi(\tilde T)$ from the proof of Theorem~\ref{theorem:fibres}.
  Suppose that $\Phi(\AA)=(\alpha,\tilde\AA)$.
  
  Plot the elements of $C$ and $\tilde C$ on the number line, and add a node labeled $\infty$ to the right of all these points.
  Join each $\tilde c\in \tilde C$ to $\alpha(\tilde c)$ (which lies to its left) by an arc.
  Finally, join the elements of $C-\alpha(\tilde C)$ to $\infty$ by arcs.
  Let $w(\alpha)$ denote the number of crossings of these arcs.
  Since these arcs are precisely the first arcs of $\AA$,
  \begin{equation}
    \label{eq:wv}
    v(\AA) = w(\alpha) + v(\tilde\AA).
  \end{equation}
  
  Furthermore, if $|\tilde C|=l$, then
  \begin{equation}
    \label{eq:w}
    \sum_{\alpha\in M(C,\tilde C)} q^{w(\alpha)} = \prod_{i=1}^l [c_{i2}(T)]_q.
  \end{equation}
  To see this, suppose the entries of $\tilde C$ are $\tilde c_1,\dotsc,\tilde c_l$ in increasing order and $\alpha$ is chosen by choosing $\alpha(\tilde c_1),\dotsc,\alpha(\tilde c_l)$ sequentially as in the proof of Theorem~\ref{theorem:fibres} (see Example~\ref{example:first-arcs} below).
  Having chosen $\alpha(\tilde c_1),\dotsc,\alpha(\tilde c_{i-1})$, we have $c_{i2}(T)$ choices for $\alpha(\tilde c_i)$.
  The $r$th choice from the right leaves $r-1$ unused nodes between $\alpha(\tilde c_i)$ and $\tilde c_i$.
  These nodes will eventually be joined to $\tilde c_s$ for some $s>i$ or to $\infty$, leading to $r-1$ crossings that will be counted in the statistic $w(\alpha)$.
  Thus the $c_{i2}(T)$ choices will contribute factors $1,q,\dotsc,q^{c_{i2}(T)-1}$ to $q^{w(\alpha)}$.
  Summing over all possible $\alpha$ gives (\ref{eq:w}).
  
  By (\ref{eq:wv}), (\ref{eq:w}), and the induction hypothesis,
  \begin{align*}
    \sum_{\T(\AA)=T} q^{v(\AA)} &  = \sum_{\alpha\in M(C,\tilde C)} q^{w(\alpha)} \sum_{\T(\tilde\AA)=\tilde T} q^{v(\tilde\AA)}\\
                                & = \prod_{\{i\mid \lambda_i\geq 2\}} [c_{i2}(T)]_q c_q(\tilde T)  = c_q(T),
  \end{align*}
  completing the proof of Theorem~\ref{theorem:statistic}.
\end{proof}
\begin{example}
  \label{example:first-arcs}
  Consider $C=1,2,3,6,9$ and $\tilde C=5,8,10,11$.
  Having chosen $\alpha(5)=2$ and $\alpha(8)=3$, the available choices for $\alpha(10)$ are represented by the dotted arcs in Figure~\ref{fig:choices}. From right to left, these choices are $9$, $6$, and $1$, and will eventually contribute $0$, $1$, and $2$ crossings respectively with the arc joining $\alpha(11)$ to $11$ and the arcs joining the elements of $C-\alpha(\tilde C)$ to infinity. 
  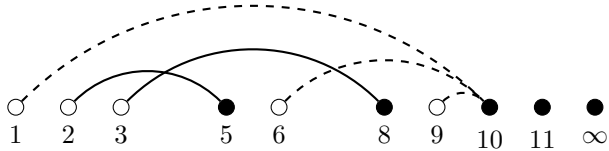
\begin{figure}[h]
    \begin{displaymath}
      \begin{tikzpicture}
        [scale=.7,every node/.style={circle,draw,inner sep=2pt, minimum size=6pt}]
        \node[label=below:$1$] (1) at (1,0) {};
        \node[label=below:$2$] (2) at (2,0) {};
        \node[label=below:$3$] (3) at (3,0) {};
        \node[label=below:$5$, fill=black] (5) at (5,0) {};
        \node[label=below:$6$] (6) at (6,0) {};
        \node[label=below:$8$, fill=black] (8) at (8,0) {};
        \node[label=below:$9$] (9) at (9,0) {};
        \node[label=below:$10$, fill=black] (10) at (10,0) {};
        \node[label=below:$11$, fill=black] (11) at (11,0) {};
        \node[label=below:$\infty$, fill=black] (0) at (12,0) {};
        \draw[thick,color=black]
        (2) [out=45, in=135] to  (5);
        \draw[thick,color=black]
        (3) [out=45, in=135] to  (8);
        \draw[thick,color=black, style=dashed]
        (9) [out=45, in=135] to  (10);
        \draw[thick,color=black, style=dashed]
        (6) [out=45, in=135, color=black, style=dashed] to  (10);
        \draw[thick,color=black, style=dashed]
        (1) [out=45, in=135] to  (10);
      \end{tikzpicture}
    \end{displaymath}
    \caption{Choices for $\alpha(10)$.}
    \label{fig:choices}
  \end{figure}
\end{example}

\begin{corollary}
  \label{cor:noninterl}
For each partition $\lambda$ of $n$, noninterlacing partitions of shape $\lambda$ are in bijection with standard tableaux of shape $\lambda$. Therefore noninterlacing partitions of $[n]$ are in bijection with involutions in the symmetric group $S_n$. 
\end{corollary}
\begin{proof}
  $c_q(T)$ as defined in Eq. \eqref{eq:rq} is a product of $q$-integers. Therefore the constant term of $c_q(T)$ is 1. Theorem \ref{theorem:statistic} implies that for every standard tableau $T$ there is a unique noninterlacing partition $\AA$ such that $\T(\AA)=T$. The second statement now follows since the Robinson-Schensted correspondence gives a bijection from involutions in the symmetric group $S_n$ to standard tableaux with $n$ cells (see Aigner \cite[Thm. 8.26]{MR2339282}).
\end{proof}
 Theorem \ref{theorem:statistic} allows us to express $b_\lambda(q)$ in Eq. \eqref{eq:bla} as a sum over set partitions of shape $\lambda$:
 \begin{equation}
   \label{eq:blamset}
b_\lambda(q)=\sum_{\AA\in \Pi_{n}(\lambda)}q^{v(\AA)}.   
 \end{equation}

\subsection{Recursive computation of $b_\lambda(q)$} 
Given a partition $\lambda=(\lambda_1,\dotsc,\lambda_m)$, its Young diagram is the collection
\begin{displaymath}
  Y(\lambda) = \{(i,j) \mid 1\leq i\leq m\text{ and } 1\leq j\leq \lambda_i\}.
\end{displaymath}
The elements of $Y(\lambda)$ are called cells.
A cell $c\in Y(\lambda)$ is said to be removable if $Y(\lambda)-\{c\}$ is again the Young diagram of a partition.
  Recall that the \emph{conjugate} of an integer partition $\lambda=(\lambda_1,\dotsc,\lambda_m)$ is the integer partition $\mu=(\mu_1,\dotsc,\mu_k)$ where $k=\lambda_1$, and for each $1\leq i\leq k$, $\mu_i$ is the number of parts of $\lambda$ that are greater than or equal to $i$.
  Visually, if $\mu = \lambda'$, then $\mu_j$ is the number of cells in the $j$th column of the Young diagram of $\lambda$.
\label{sec:rec-comp}
\begin{theorem}
  \label{theorem:rec-bla}
  Let $\lambda=(\lambda_1,\dotsc,\lambda_m)$ be a partition, and $\mu=(\mu_1,\dotsc,\mu_k)$ denote its conjugate.
  Suppose the removable cells of $Y(\lambda)$ are
  \begin{displaymath}
    (i_1,j_1),\dotsc,(i_r,j_r).
  \end{displaymath}
  Let $\lambda^-_s$ be the partition obtained from $\lambda$ upon removal of the cell $(i_s,j_s)$ for each $1\leq s\leq r$.
  Then
  \begin{equation}
    \label{eq:bla-rec}
    b_\lambda(q) = \sum_{s=1}^r [\mu_{j_s-1}-(i_s-1)]_qb_{\lambda^-_s}(q),
  \end{equation}
  where $\mu_{j_s-1}-(i_s-1)$ should be interpreted as $1$ if $j_s=1$.
\end{theorem}
\begin{proof}
  Suppose that $\lambda$ is a partition of $n$.
  Given $T\in \Tab_{[n]}(\lambda)$, the cell $c=(i,j)$ that contains $n$ is removable, and $c_{ij}(T)=\mu_{j-1}-(i-1)$.
  The tableau $T^-$ obtained by deleting $n$ from $T$ is a standard tableau with support $[n-1]$, and $c_q(T)=c_q(T^-)[\mu_{j-1}-(i-1)]_q$.
  Summing over all $T\in \Tab_{[n]}(\lambda)$ gives the desired identity.
\end{proof}
\begin{example}
If $\lambda = (4,3,3,1)$, then the removable cells are
\begin{displaymath}
  (1,4),(3,3),(4,1),
\end{displaymath}
indicated by $X$ in the diagram
\begin{displaymath}
  \ytableaushort{{}{}{}X,{}{}{},{}{}X,X},
\end{displaymath}
and the partitions obtained by removing them are
\begin{displaymath}
  (3,3,3,1), (4,3,2,1), (4,3,3).
\end{displaymath}
We get
\begin{displaymath}
  b_{(4,3,3,1)}(q) = [3]_q b_{(3,3,3,1)} + [1]_q b_{(4,3,2,1)} +  b_{(4,3,3)}.
\end{displaymath}
\end{example}
With base case $b_\lambda(q)=1$ if $\lambda=\emptyset$ (the empty partition of $0$), (\ref{eq:bla-rec}) gives a recursive algorithm to compute $b_\lambda(q)$.
However, unraveling the recursion tree in (\ref{eq:bla-rec}) will just lead to the sum \eqref{eq:bla} over all standard tableaux.
If the previously computed steps are cached, then some repetition can be avoided.
Also, if the objective is to compute $b_\lambda(q)$ for all partitions of size up to a given integer, then this recursion is much faster than using \eqref{eq:bla}.
It was possible to compute $b_\lambda(q)$ for all partitions $\lambda$ of size up to $50$ in a few hours on an ordinary office desktop using this algorithm.

\subsection{$q$-Stirling numbers of the second kind}
\begin{definition}
  \label{definition:q-stirling}
  The $q$-Stirling numbers of the second kind, $S_q(n,m)$ for $n,m\geq 0$, are defined by
  \begin{displaymath}
    S_q(n,m) = S_q(n-1,m-1) + [m]_q S_q(n-1,m) \text{ for all } n\geq 1, m\geq 1,
  \end{displaymath}
  and $S_q(n,m)=\delta_{nm}$ (Kronecker delta) when either $n=0$ or $m=0$.
\end{definition}
For each partition $\lambda$, let $l(\lambda)$ denote the number of parts of $\lambda$.
\begin{theorem}
  \label{theorem:q-stirling}
  For all $n,m\geq 0$,
  \begin{displaymath}
    S_q(n,m) = \sum_{\lambda \vdash n,\;l(\lambda)=m}q^{\sum_i (i-1)(\lambda_i-1)}\sum_{T\in \Tab_{[n]}(\lambda)} c_q(T).
  \end{displaymath}
\end{theorem}
\begin{proof}
  Let $\mu=\lambda'$, the integer partition conjugate to $\lambda$.
  We have $\sum_{i=1}^m(i-1)(\lambda_i-1) = \sum_{j=2}^k\binom{\mu_j}2$. Indeed, both sides represent the sum of all the entries of the tableau of shape $\lambda$ whose first column is filled with $0$'s and the remaining entries of the $i$th row are  $i-1$ for each $i\geq 1$.
  For example if $\lambda = (5,4,4,2)$, then this tableau is
  \begin{displaymath}
    \ytableaushort{00000,0111,0222,03}.
  \end{displaymath}

  In order to prove Theorem~\ref{theorem:q-stirling}, it suffices to show that the sum
  \begin{displaymath}
    \Sigma_q(n,m) := \sum_{\lambda\vdash n,\;l(\lambda)=m}q^{\sum_i(i-1)(\lambda_i-1)}\sum_{T\in \Tab_{[n]}(\lambda)} c_q(T)
  \end{displaymath}
  satisfies the defining conditions for $S_q(n,m)$ given in Definition \ref{definition:q-stirling}.

  Note that $\Sigma_q(0,0)=1$, since the empty partition of $0$ has length $0$, and $c_q(T)=1$ for the the empty tableau $T$.
  On the other hand if precisely one of $n$ and $m$ is zero, then $\Sigma_q(n,m)=0$ since the outer sum is empty.

  It remains to show that
  \begin{displaymath}
    \Sigma_q(n,m) := \Sigma_q(n-1,m-1) + [m]_q \Sigma_q(n-1,m) \text{ for all } n\geq 1, m\geq 1.
  \end{displaymath}
  The left hand side is a sum over standard tableaux with $n$ cells in $m$ rows.
  We will see that the tableaux which contain $n$ in their first column contribute $\Sigma_q(n-1,m-1)$ to this sum, and the remaining tableaux contribute $[m]_q\Sigma_q(n-1,m)$.
  
  First, suppose that $n$ lies in the first column of $T$.
  Removing $n$ from $T$ results in a tableau $\tilde T$ with $n-1$ cells in $m-1$ rows.
  Moreover, $c_q(T)=c_q(\tilde T)$.
  Therefore the contribution of such tableaux to $\Sigma_q(n,m)$ is $\Sigma_q(n-1,m-1)$.

  Now we come to the contribution of tableaux $T$ where $n$ does not occur in the first column.
  For each tableau $\tilde T$ with $n-1$ cells in $m$ rows, let $U(\tilde T)$ be the set of all standard tableaux $T$ with $n$ cells in $m$ rows such that $\tilde T$ is obtained when the cell containing $n$ is removed from $T$.
  Suppose that the shape of $\tilde T$ is $\tilde\lambda$, and $\tilde\lambda'=\tilde\mu=(\tilde\mu_1,\dotsc,\tilde\mu_k)$.
  \begin{figure}[h]
    \centering
    \includegraphics[width=0.5\textwidth]{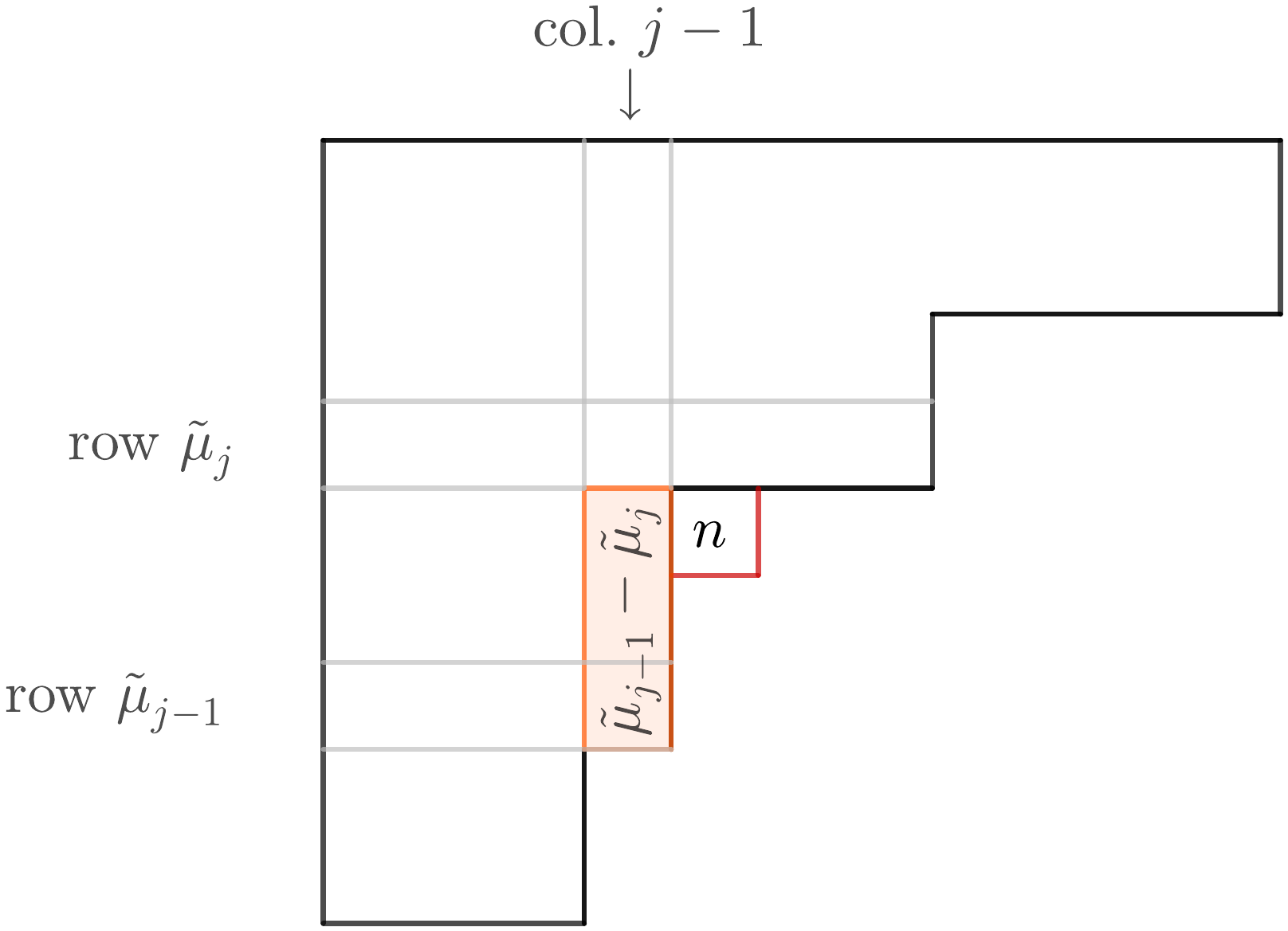}\quad
    \includegraphics[width=0.41\textwidth]{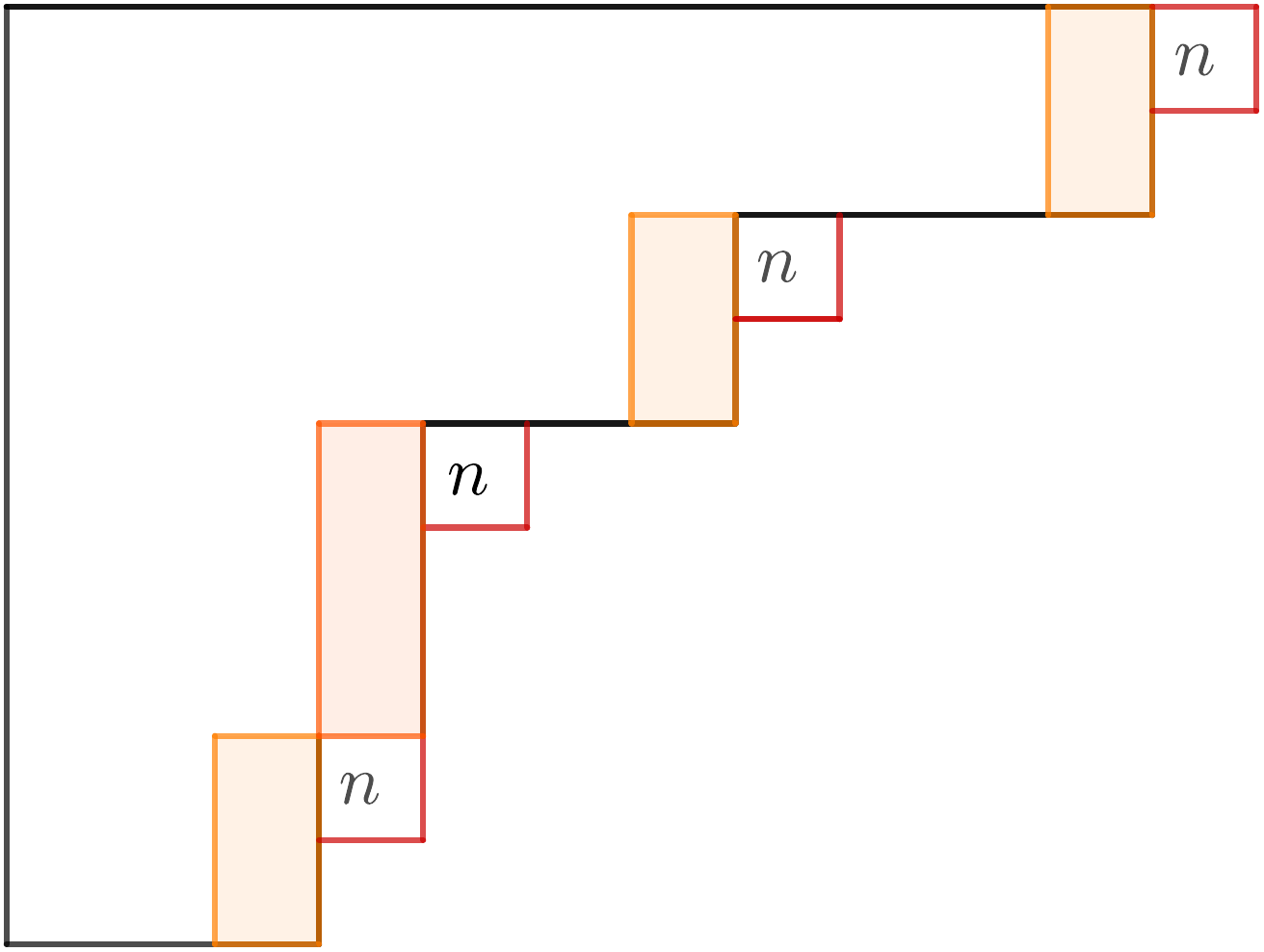}
    \caption{Left: Cell containing $n$ in the $j$th column.\\  Right: Contributions of all tableaux in $U(\tilde T)$}
    \label{fig:removal}
  \end{figure}

  When the cell in $T$ containing $n$ lies in its $j$th column (see Figure~\ref{fig:removal}), then we must have $\tilde\mu_j<\tilde\mu_{j-1}$.
  We have $c_{\tilde\mu_j+1,j}(T)=\tilde\mu_{j-1}-\tilde\mu_j$ (counting the highlighted cells in Figure~\ref{fig:removal}).
  Let $\lambda$ denote the shape of $T$.
  Then $\lambda'=\mu=(\tilde\mu_1,\dots,\tilde\mu_{j-1},\tilde\mu_j+1,\tilde\mu_{j+1},\dotsc,\tilde\mu_k)$.
  We have
  \begin{displaymath}
    \sum_i(i-1)(\lambda_i-1) = \sum_{j\geq 2} \binom{\mu_j}2 = \tilde\mu_j + \sum_{j\geq 2}\binom{\tilde\mu_j}2 = \tilde\mu_j + \sum_i (i-1)(\tilde\lambda_i-1).
  \end{displaymath}
  Therefore,
  \begin{align*}
    q^{\sum_i(i-1)(\lambda_i-1)}c_q(T) &= q^{\tilde\mu_j}[\tilde\mu_{j-1}-\tilde\mu_j]_qq^{\sum_i(i-1)(\tilde\lambda_i-1)}c_q(\tilde T)\\&=(q^{\tilde\mu_j}+\dotsb+q^{\tilde\mu_{j-1}-1})q^{\sum_i(i-1)(\tilde\lambda_i-1)}c_q(\tilde T).
  \end{align*}
  Each tableau in $U(\tilde T)$ is obtained by adding a cell containing $n$ to the $j$th column of $\tilde T$ for some $j$ such that $\tilde\mu_j<\tilde\mu_{j-1}$.
  Adding up the contributions of all such tableaux (Figure~\ref{fig:removal}) gives
  \begin{align*}
    \sum_{T\in U(\tilde T)} q^{\sum_i(i-1)(\lambda_i-1)}c_q(T) & = [\tilde\mu_1]_qq^{\sum_i(i-1)(\tilde\lambda_i-1)}c_q(\tilde T)\\&= [m]_qq^{\sum_i(i-1)(\tilde\lambda_i-1)}c_q(\tilde T).
  \end{align*}
  Each tableau $T$ where $n$ does not lie in the first column belongs to $U(\tilde T)$ for a unique standard tableau $\tilde T$ of size $n-1$ with $m$ elements in its first column.
  It follows that these tableaux collectively contribute $[m]_q\Sigma_q(n-1,m)$ to the sum defining $\Sigma_q(n,m)$.
\end{proof}
\begin{example} For
  \begin{displaymath}
    \tilde T = \ytableaushort{125,34{10},68,79},
  \end{displaymath}
  $U(\tilde T)$ consists of the tableaux
  \begin{displaymath}
    T_1 = \ytableaushort{125,34{10},68{11},79}\text{ and } T_2=\ytableaushort{125{11},34{10},68,79}.
  \end{displaymath}
  We have
  \begin{displaymath}
    q^9c_q(T_1) = q^2[2]_q\times q^7c_q(\tilde T), \text{ and } q^7c_q(T_2) = q^0[2]_q\times q^7c_q(\tilde T).
  \end{displaymath}
  Therefore,
  \begin{displaymath}
    q^9c_q(T_1)+q^7c_q(T_2) = q^2(1+q)q^7c_q(\tilde T) + (1+q)q^7c_q(\tilde T) = [4]_qq^7c_q(\tilde T).
  \end{displaymath}
\end{example}
Let $\Pi_{n,m}$ denote the set of all partitions of $[n]$ with $m$ parts.
For $\AA\in \Pi_{n,m}$ let $(\lambda_1^\AA,\dots,\lambda_m^\AA)$ denote the shape of $\AA$.
The expression for $S_q(n,m)$ in Theorem~\ref{theorem:q-stirling} can be written as a sum over elements of $\Pi_{n,m}$ of monomials using Theorem~\ref{theorem:statistic}:
\begin{equation}
  \label{eq:q-stirling}
  S_q(n,m) = \sum_{\AA\in \Pi_{n,m}} q^{v(\AA)+\sum_i(i-1)(\lambda_i^\AA-1)}.
\end{equation}
In view of Corollary \ref{cor:noninterl}, $S_q(n,m)$ can also be written as a sum over a much smaller class of set partitions 
\begin{equation}
  \label{eq:qs4}
S_q(n,m)=\sum_{ \substack{\AA\in\Pi_{n,m}\\v(\AA)=0}}q^{\sum_i (i-1)(\lambda_i^\AA-1)}c_q(\T(\AA)).  
\end{equation} 
The trade-off is that the summand is now a power of $q$ times a product of $q$-integers.
In terms of the polynomials $b_\lambda(q)$, the identity of Theorem~\ref{theorem:q-stirling} has an alternate form:
\begin{equation}
  \label{eq:q-stirling-alt}
  S_q(n,m) = \sum_{\lambda\vdash n,\;l(\lambda)=m} q^{\sum_i(i-1)(\lambda_i-1)}b_\lambda(q).
\end{equation}

\subsection{Value at $q=-1$}
\label{sec:minusone}
It turns out that $b_\lambda(-1)$ is always positive.
\begin{theorem}
  For every integer partition $\lambda$, $b_\lambda(-1)$ is the number of standard tableaux $T$ of shape $\lambda$ for which $c(T)$ is odd.
  In particular, for every non-empty\footnote{The empty partition is the unique partition of $0$ with no parts.} integer partition $\lambda$, $b_\lambda(-1)>0$.
\end{theorem}
\begin{proof}
  For every $n\in \nat$, the value of a $q$-integer at $q=-1$ is
  \begin{displaymath}
    [n]_q|_{q=-1} =
    \begin{cases}
      1 & \text{if } n \text{ is odd},\\
      0 & \text{if } n \text{ is even}.
    \end{cases}
  \end{displaymath}
  Hence $c_q(T)$ is zero at $q=-1$ except when all the integers $c_{ij}(T)$ appearing in (\ref{eq:rq}) are odd, in which case it is equal to $1$.
  It follows from the definition of $b_\lambda(q)$ (\ref{eq:bla}) that $b_\lambda(-1)$ is the number of tableaux $T$ with shape $\lambda$ and $c(T)$ odd.

  To prove that $b_\lambda(-1)>0$, it suffices to show that for each non-empty integer partition $\lambda$ of $n$, there is at least one tableau $T$ of shape $\lambda$ such that $c(T)$ is odd.
  In fact, there is always a tableau $T$ of shape $\lambda$ with $c(T)=1$; just enter the integers from $1$ to $n$ in increasing order along the rows, starting with the top row and going down (see Example~\ref{example:one-tab} below).
\end{proof}
In the spirit of Stemberidge's $q=-1$ phenomenon~\cite{ste}, it would be interesting to find a natural involution on the collection of set partitions that changes the parity of the number of interlacings, and leaves fixed only noninterlacing set partitions corresponding to tableaux $T$ for which $c(T)$ is odd.
\begin{example}
  \label{example:one-tab}
  For $\lambda = (3,2,2)$, the tableau $T=\ytableaushort{123,45,67}$ has $c(T)=1$.
\end{example}
For a large class of partitions $\lambda$, $b_\lambda(-1)$ has an interpretation in terms of shifted tableaux.
Let $\lambda=(\lambda_1,\dotsc,\lambda_m)$ be an integer partition.
A shifted tableau of shape $\lambda$ is an array
\begin{displaymath}
  T = (T_{ij}\mid T_{ij}\in \nat,\;1\leq i \leq m;\; i\leq j\leq \lambda_i+i-1).
\end{displaymath}
\begin{definition}
  A multilinear shifted tableau is a shifted tableau $T$ where each integer occurs at most once, the rows of $T$ increase from left to right, the columns of $T$ increase from top to bottom, and the diagonals of $T$ increase from top-left to bottom-right.
  The set $S$ of entries of $T$ is called its support.
  A shifted tableau with support $[n]$ is called a standard shifted tableau.
  We denote the set of multilinear shifted tableaux with support $S$ and shape $\lambda$ by $\stab_S(\lambda)$.
\end{definition}
\begin{example}
  The following is a standard shifted tableau of shape $(3,3,2,1)$.
  \begin{displaymath}
    \ytableaushort{123,{\none}456,\none\none 78,\none\none\none 9}
  \end{displaymath}
\end{example}

\begin{remark}
  The study of standard shifted tableaux of shape $\lambda$ when $\lambda$ has distinct parts goes back to Schur. See the survey article of Adin and Roichman~\cite[Section. 14.5.3]{MR3409355}.  Recently there has been an interest in the enumeration of shifted tableaux for partitions whose parts are not necessarily distinct \cite{MR3017956,MR3665574}. 
\end{remark}
\begin{theorem}
  \label{theorem:minusone}
  For every partition $\lambda$ of $n$ whose parts are distinct, with the possible exception of the largest part,
  \begin{displaymath}
    b_\lambda(-1) = |\stab_{[n]}(\lambda)|.
  \end{displaymath}
\end{theorem}
\begin{example}
  Take $\lambda = (3,3,1)$.
  Then
  \begin{displaymath}
    b_\lambda(q) = q^{4} + 5 q^{3} + 15 q^{2} + 28 q + 21,
  \end{displaymath}
  and $b_\lambda(-1)=4$.
  There are four standard shifted tableaux of shape $(3,3,1)$, namely
  \begin{displaymath}
    \ytableaushort{123,\none 456,\none\none 7},
    \ytableaushort{124,\none 356,\none\none 7},
    \ytableaushort{123,\none 457,\none\none 6}, \text{ and }
    \ytableaushort{124,\none 357,\none\none 6}.
  \end{displaymath}
\end{example}
\begin{proof}
  [Proof of Theorem~\ref{theorem:minusone}]
  Let $\lambda=(\lambda_1,\dotsc,\lambda_m)$ be an integer partition.
  Suppose $T\in \Tab_{[n]}(\lambda)$ is such that $c(T)$ is odd.
  This happens if and only if $c_{ij}(T)$ is odd for all $1\leq i\leq m$, $2\leq j\leq \lambda_i$.
  For each $j\geq 2$, let $s_{ij}(T):=c_{ij}(T)+(i-1)$.
  Then $s_{ij}(T)$ is the number of entries of the $(j-1)$st column of $T$ that are less than $T_{ij}$.
  Since the $j$th column of $T$ is increasing, we have
  \begin{displaymath}
    1\leq s_{1j}(T)\leq  s_{2j}(T)\leq \dotsc\leq  s_{\mu_j,j}(T)\leq \mu_{j-1},
  \end{displaymath}
  where $\mu=(\mu_1,\dotsc,\mu_k)$ is the conjugate of $\lambda$.
  Since $c_{ij}(T)$ is odd for each $1\leq i\leq \mu_j$, we must have $s_{ij}(T)$ and $s_{i+1,j}(T)$ differ by an odd number for $i=1,\dotsc,\mu_j$.
  In particular, $s_{ij}(T)<s_{i+1,j}(T)$ and if they do not differ by $1$, then they differ by at least $3$.
  It follows that $s_{\mu_jj}(T)\geq \mu_j$, and if strict inequality holds, then $s_{\mu_jj}(T)\geq \mu_j+2$.

  The assertion that the parts of $\lambda$, except possibly the largest ones, are distinct is equivalent to the assertion that successive parts of $\mu$ differ by at most one.
  Therefore $\mu_{j-1}\leq \mu_j+1$.
  It follows that $s_{\mu_j,j}(T)\leq\mu_j+1$.
  But this can only happen when $s_{ij}(T)=i$ for each $i=1,\dotsc,\mu_j$.
  This means that the entries of the $j$th column of $T$ interleave the first $\mu_j$ entries of the $(j-1)$st column:
  \begin{displaymath}
    T_{1,j-1}<T_{1j}<T_{2,j-1}<T_{2j}<\dotsb.
  \end{displaymath}
  In particular, for each $2\leq j\leq k$, $1\leq i\leq \mu_i$, we have
  \begin{displaymath}
    T_{ij}<T_{i+1,j-1}
  \end{displaymath}
  If we were to shift the $i$th row one cell to the right relative to the $(i-1)$st, each entry in the $i$th row would still be greater than the corresponding entry in the $(i-1)$st row.
  Doing this with each row would result in a standard shifted tableau of shape $\lambda$.
\end{proof}
\begin{example}
  The tableau
  \begin{displaymath}
    T = \ytableaushort{124,356,7}
  \end{displaymath}
  of shape $(3,3,1)$ has $c(T)=1$.
  Shifting each row one cell to the right relative to row one above it still results in increasing columns and diagonals:
  \begin{displaymath}
    \ytableaushort{124,\none 356,\none\none 7}.
  \end{displaymath}
\end{example}

\subsection{Relation to the $q$-Hermite Catalan matrix}
\label{sec:relation-q-hermite}

In the combinatorial theory of orthogonal polynomials developed by Viennot \cite{viennot1983theorie} (see  also \cite[Section~7.1]{MR2339282} and \cite{MR3409351}), given sequences $\{b_k\}_{k\in \nn}$ and $\{\lambda_k\}_{k\in \nat}$ in a unital commutative ring,  the associated Catalan triangle $A=(a_{n,k})_{k, n\in \nn}$ is defined by
\begin{align*}
  a_{0,0} & = 1,\\
  a_{0,k} & = 0, \text{ for } k\geq 1,\\
  a_{n,k} & = a_{n-1,k-1} + b_k a_{n-1,k} + \lambda_{k+1}a_{n-1,k+1} \text{ for }n\geq 1.
\end{align*}
In the last identity, if $k=0$, the term $a_{n-1,k-1}$ is taken as $0$.
The entries $a_{n,0}$ are called generalized Catalan numbers associated to $\{b_k\}$ and $\{\lambda_k\}$, and are the moments of the family of orthogonal polynomials associated to these sequences by Favard's theorem.
The case $b_k=0$ and $\lambda_k=1$ for all $k$ corresponds to Chebychev polynomials of the second kind, where $a_{n,0}$ is $0$ for odd $n$, and equals the Catalan number $C_m$ for $n=2m$.

The case $b_k=0$ for all $k$ and $\lambda_k=k$ corresponds to Hermite polynomials, while the case $b_k=0$ and $\lambda_k=[k]_q$ corresponds to $q$-Hermite polynomials~\cite{MR930175}.
Since $b_k=0$ for all $k$, $a_{n,k}=0$ if $n+k$ is odd.
Given $n,k\in \nn$ such that $n+k$ is even, let $\lambda(n,k)$ denote the partition $(2^{(n-k)/2},1^k)$.
Then by Theorem~\eqref{eq:bla-rec}, we have
\begin{displaymath}
  b_{\lambda(n,k)}(q) = b_{\lambda(n-1,k-1)}(q) + [k+1]_qb_{\lambda(n-1,k+1)}.
\end{displaymath}
Thus $b_{\lambda(n,k)}(q)$ is the entry $a_{n,k}$ of the Catalan triangle associated to $q$-Hermite polynomials.

In particular, for every $m\in \nn$, $b_{(2^m)}(q)$ is the $2m$th moment of the $q$-Hermite orthogonal polynomial sequence.
Consider the expression \eqref{eq:blamset} for $b_{(2^m)}(q)$.
The set $\Pi_{n}(2^m)$ coincides with the set of chord diagrams on $2m$ points and the interlacing number coincides with the number of crossing chords.
Thus $b_{(2^m)}(q)$ is the polynomial $T_m(q)$ studied by Touchard \cite{MR37815,MR36006,MR46325}  in the context of the stamp folding problem.
Based on Touchard's work, Riordan~\cite{MR366686} gave a compact expression for $T_m(q)$:
\begin{equation}
  \label{eq:touchard-riordan}
  T_m(q)(1-q)^m = \sum_{i=0}^m (-1)^i\left[\binom{2m}{m-i}-\binom{2m}{m-i-1}\right]q^{\binom{i+1}2}.
\end{equation}
See Read~\cite{MR556055} for an alternative proof of this formula using continued fractions and Penaud~\cite{MR1336847} for a bijective proof.
A beautiful exposition can be found in Aigner \cite[p. 337]{MR2339282}.      

\subsection{Partitions with two parts}
\label{sec:twoparts}
\begin{theorem}
  For integers $r\geq s\geq 0$,
  \begin{displaymath}
    b_{(r,s)}(q) =
    \begin{cases}
      \sum_{i=0}^s f_{(r+i,s-i)}q^i &\text{if } r>s,\\
      \sum_{i=0}^{s-1} f_{(r+i,s-1-i)}q^i&\text{if } r=s.
    \end{cases}
  \end{displaymath}
  Here, for each partition $\lambda$, $f_\lambda$ denotes the number of standard tableaux of shape $\lambda$.
\end{theorem}
\begin{remark}
  For all integers $r\geq s\geq0$,
  \begin{align*}
    f_{(r,s)}=\frac{r+1-s}{r+s+1}{r+s+1 \choose s},
  \end{align*}
  see \cite[Cor. 10.3.2]{MR3409351}.
\end{remark}
\begin{proof}
  Note that $b_{(r,r)}(q) = b_{(r,r-1)}(q)$, since, for every tableau $T$ of shape $(r,r)$, $c_q(T)$ remains unchanged when the cell containing $2r$ is deleted from $T$.
  Therefore, it suffices to consider the case where $r>s$.
  By Eq.~\eqref{eq:blamset}, we need to show that set partitions of shape $(r,s)$ with interlacing number $i$ are in bijection with standard tableaux of shape $(r+i,s-i)$.
  Such a bijection is provided by Theorem~\ref{theorem:two-row-bij}.
\end{proof}
The bijection rests on the following lemma.
\begin{lemma}
  \label{lemma:decomposition}
  Let $\AA=a_1\dotsb a_l|b_1\dotsb b_m$ be a set partition such that the $j$th arcs of the two blocks cross.
  Then $\{a_1,\dotsc,a_j\}\cup \{b_1,\dotsc,b_j\}=[2j]$.
\end{lemma}
\begin{proof}
  If the $j$th arcs cross, then either $a_j<b_j<a_{j+1}<b_{j+1}$, or $b_j<a_j<b_{j+1}<a_{j+1}$.
  In either case, each of $a_j$ and $b_j$ is less than both $a_{j+1}$ and $b_{j+1}$.
  Thus all the elements of $\{a_1,\dotsc,a_j\}\cup\{b_1,\dotsc,b_j\}$ are less than the remaining elements of $[r+s]$, and hence must be the elements of $[2j]$ (see the first row of Figure~\ref{fig:steps}).
\end{proof}
\begin{theorem}
  \label{theorem:two-row-bij}
  There exists a bijection $\Theta$ that maps the collection of set partitions of shape $(r,s)$ with interlacing number $i$ bijectively onto the set of standard Young tableaux of shape $(r+i,s-i)$ when $r>s$.
\end{theorem}
\begin{proof}
  The construction of $\Theta$ involves the following five steps, illustrated by Figure~\ref{fig:steps}.
  \begin{figure}[tb]
    \begin{displaymath}
      \xymatrix{
        \begin{tikzpicture}
          [scale=0.7,every node/.style={circle,fill=black, scale=0.4}]
          \tikzset{minimum size=2pt}
          \node[label=below:$1$] (1) at (1,0) {};
          \node[label=below:$2$] (2) at (2,0) {};
          \node[label=below:$3$] (3) at (3,0) {};
          \node[label=below:$4$] (4) at (4,0) {};
          \node[label=below:$5$] (5) at (5,0) {};
          \node[label=below:$6$] (6) at (6,0) {};
          \node[label=below:$7$] (7) at (7,0) {};
          \node[label=below:$8$] (8) at (8,0) {};
          \node[label=below:$9$] (9) at (9,0) {};
          \node[label=below:$10$] (10) at (10,0) {};
          \node[label=below:$11$] (11) at (11,0) {};
          \node[label=below:$12$] (12) at (12,0) {};
          \node[label=below:$13$] (13) at (13,0) {};
          \node[label=below:$14$] (14) at (14,0) {};
          \node[label=below:$\infty$] (0) at (15,0) {};
          \draw[thick,color=teal]
          (1) [out=45, in=135] to  (2);
          \draw[thick,color=orange]
          (2) [out=45, in=135] to  (6);
          \draw[thick,color=blue]
          (6) [out=45, in=135] to  (8);
          \draw[thick,color=teal]
          (8) [out=45, in=135] to  (9);
          \draw[thick,color=orange]
          (9) [out=45, in=135] to  (11);
          \draw[thick,color=blue]
          (11) [out=45, in=135] to  (12);
          \draw[thick,color=teal]
          (12) [out=45, in=135] to  (14);
          \draw[thick,color=teal]
          (3) [out=45, in=135] to  (4);
          \draw[thick,color=orange]
          (4) [out=45, in=135] to  (5);
          \draw[thick,color=blue]
          (5) [out=45, in=135] to  (7);
          \draw[thick,color=teal]
          (7) [out=45, in=135] to  (10);
          \draw[thick,color=orange]
          (10) [out=45, in=135] to  (13);
          \draw[thick,color=blue]
          (13) [out=45, in=135] to  (0);
          \draw (6.5,-1) -- (6.5,1);
          \draw (10.5,-1) -- (10.5,1);
        \end{tikzpicture}\ar[d]^(.2){1}\\
        \begin{tikzpicture}
          [scale=0.7,every node/.style={circle,fill=black, scale=0.4}]
          \tikzset{minimum size=2pt}
          \node[label=below:$1$] (1) at (1,0) {};
          \node[label=below:$2$] (2) at (2,0) {};
          \node[label=below:$3$] (3) at (3,0) {};
          \node[label=below:$4$] (4) at (4,0) {};
          \node[label=below:$5$] (5) at (5,0) {};
          \node[label=below:$6$] (6) at (6,0) {};
          \node[label=below:$1$] (7) at (7,0) {};
          \node[label=below:$2$] (8) at (8,0) {};
          \node[label=below:$3$] (9) at (9,0) {};
          \node[label=below:$4$] (10) at (10,0) {};
          \node[label=below:$1$] (11) at (11,0) {};
          \node[label=below:$2$] (12) at (12,0) {};
          \node[label=below:$3$] (13) at (13,0) {};
          \node[label=below:$4$] (14) at (14,0) {};
          \node[label=below:$\infty$] (0) at (15,0) {};
          \draw[thick,color=teal]
          (1) [out=45, in=135] to  (2);
          \draw[thick,color=orange]
          (2) [out=45, in=135] to  (6);
          \draw[thick,color=teal]
          (8) [out=45, in=135] to  (9);
          \draw[thick,color=blue]
          (11) [out=45, in=135] to  (12);
          \draw[thick,color=teal]
          (12) [out=45, in=135] to  (14);
          \draw[thick,color=teal]
          (3) [out=45, in=135] to  (4);
          \draw[thick,color=orange]
          (4) [out=45, in=135] to  (5);
          \draw[thick,color=teal]
          (7) [out=45, in=135] to  (10);
          \draw[thick,color=blue]
          (13) [out=45, in=135] to  (0);
          \draw (6.5,-1) -- (6.5,1);
          \draw (10.5,-1) -- (10.5,1);
        \end{tikzpicture}\ar[d]^{2}\\
        \ytableausetup{smalltableaux}
        \ytableaushort{125,346}\quad \ytableaushort{13,24} \quad \ytableaushort{124,3}\ar[d]^3\\
        \begin{tikzpicture}[scale=0.3]
          \filldraw(0,0) circle (2pt);
          \draw (0,0)--(1,1);
          \filldraw (1,1) circle (2pt);
          \filldraw(0,0) circle (2pt);
          \draw (1,1)--(2,2);
          \filldraw (2,2) circle (2pt);
          \filldraw(0,0) circle (2pt);
          \draw (2,2)--(3,1);
          \filldraw (3,1) circle (2pt);
          \filldraw(0,0) circle (2pt);
          \draw (3,1)--(4,0);
          \filldraw (4,0) circle (2pt);
          \filldraw(0,0) circle (2pt);
          \draw (4,0)--(5,1);
          \filldraw (5,1) circle (2pt);
          \filldraw(0,0) circle (2pt);
          \draw[color=red] (5,1)--(6,0);
          \filldraw (6,0) circle (2pt);
        \end{tikzpicture}
        \quad
        \begin{tikzpicture}[scale=0.3]
          \filldraw(0,0) circle (2pt);
          \draw (0,0)--(1,1);
          \filldraw (1,1) circle (2pt);
          \filldraw(0,0) circle (2pt);
          \draw (1,1)--(2,0);
          \filldraw (2,0) circle (2pt);
          \filldraw(0,0) circle (2pt);
          \draw (2,0)--(3,1);
          \filldraw (3,1) circle (2pt);
          \filldraw(0,0) circle (2pt);
          \draw[color=red] (3,1)--(4,0);
          \filldraw (4,0) circle (2pt);
        \end{tikzpicture}
        \quad
        \begin{tikzpicture}[scale=0.3]
          \filldraw(0,0) circle (2pt);
          \draw (0,0)--(1,1);
          \filldraw (1,1) circle (2pt);
          \filldraw(0,0) circle (2pt);
          \draw (1,1)--(2,2);
          \filldraw (2,2) circle (2pt);
          \filldraw(0,0) circle (2pt);
          \draw (2,2)--(3,1);
          \filldraw (3,1) circle (2pt);
          \filldraw(0,0) circle (2pt);
          \draw (3,1)--(4,2);
          \filldraw (4,2) circle (2pt);
        \end{tikzpicture}\ar[dd]^(.25)4\\
        \\
        \begin{tikzpicture}[scale=0.3]
          \filldraw(0,0) circle (2pt);
          \draw (0,0)--(1,1);
          \filldraw (1,1) circle (2pt);
          \filldraw(0,0) circle (2pt);
          \draw (1,1)--(2,2);
          \filldraw (2,2) circle (2pt);
          \filldraw(0,0) circle (2pt);
          \draw (2,2)--(3,1);
          \filldraw (3,1) circle (2pt);
          \filldraw(0,0) circle (2pt);
          \draw (3,1)--(4,0);
          \filldraw (4,0) circle (2pt);
          \filldraw(0,0) circle (2pt);
          \draw (4,0)--(5,1);
          \filldraw (5,1) circle (2pt);
          \filldraw(0,0) circle (2pt);
          \draw[color=red] (5,1)--(6,2);
          \filldraw (6,2) circle (2pt);
          \filldraw(0,0) circle (2pt);
          \draw (6,2)--(7,3);
          \filldraw (7,3) circle (2pt);
          \filldraw(0,0) circle (2pt);
          \draw (7,3)--(8,2);
          \filldraw (8,2) circle (2pt);
          \filldraw(0,0) circle (2pt);
          \draw (8,2)--(9,3);
          \filldraw (9,3) circle (2pt);
          \filldraw(0,0) circle (2pt);
          \draw[color=red] (9,3)--(10,4);
          \filldraw (10,4) circle (2pt);
          \filldraw(0,0) circle (2pt);
          \draw (10,4)--(11,5);
          \filldraw (11,5) circle (2pt);
          \filldraw(0,0) circle (2pt);
          \draw (11,5)--(12,6);
          \filldraw (12,6) circle (2pt);
          \filldraw(0,0) circle (2pt);
          \draw (12,6)--(13,5);
          \filldraw (13,5) circle (2pt);
          \filldraw(0,0) circle (2pt);
          \draw (13,5)--(14,6);
          \filldraw (14,6) circle (2pt);
        \end{tikzpicture}\ar[d]^5\\
        \ytableaushort{125679{10}{11}{12}{14},348{13}}
      }
    \end{displaymath}
    \caption{The five-step construction of $\Theta$.}
    \label{fig:steps}
  \end{figure}
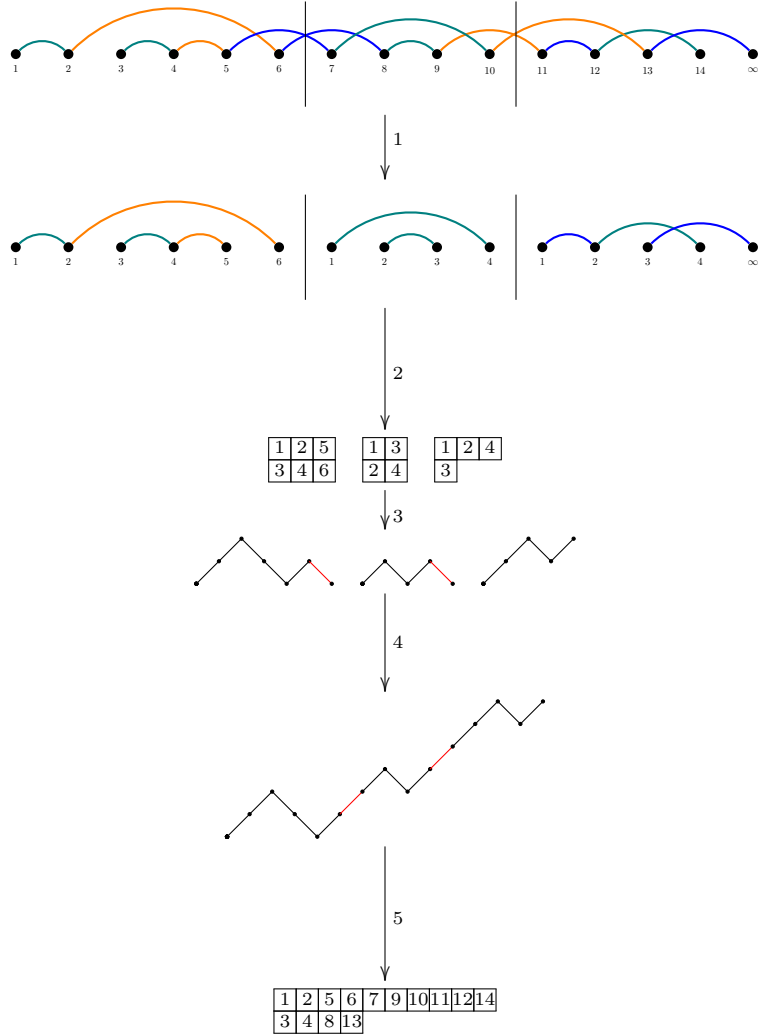
  Begin with the chord diagram of a set partition of shape $(r,s)$ and interlacing number $i$.
  \renewcommand{\thesubsubsection}{\arabic{subsubsection}}
  \subsubsection{Cut chord diagram at interlacings} 
  By Lemma~\ref{lemma:decomposition}, $i+1$ segments are obtained, each consisting of a contiguous set of integers.
  If a segment has size $l$ then renumber its nodes $1,\dotsc,l$.
  Each segment represents a non-interlacing set partition with two blocks of the same size, with the exception of the last segment.
  To reverse this process, attach the block containing the largest element of each segment to the block with the larger least element of the next segment.
  \subsubsection{Replace each non-interlacing set partition by its tableau}
  By Corollary~\ref{cor:noninterl}, $\mathcal A\mapsto \T(\mathcal A)$ is a bijective correspondence from non-interlacing set partitions to standard tableaux.
  \subsubsection{Replace each tableau with its (prefix) Dyck path}
  Given a two-row tableau $T$ of size $n=r+s$, if $j\in [n]$ occurs in the first (resp. second) row of $T$, then the $j$th step of the associated path is NE (resp. SE).
  This gives rise to a bijection from two-row standard tableau to prefix Dyck paths.
  Tableaux with rows of the same length give rise to Dyck paths.
  Thus a sequence of $i$ non-empty Dyck paths, followed by a non-empty prefix Dyck path is obtained.
  \subsubsection{Change the last step of each Dyck path from SE to NE and concatenate}
  In the first $i$ Dyck paths, the last step is always SE.
  Replace it with a NE step.
  Concatenate the resulting $i+1$ prefix Dyck paths, and record the number $i$.
  The resulting prefix Dyck path will end at height $2i+r-s$.

  This step can be reversed if we know the number $i$ of Dyck paths in the original sequence of Dyck paths.
  For $1\leq j\leq i$ look for the right-most NE step going from level $2j-1$ to level $2j$.
  Cut the prefix Dyck path after each of these steps, and change the NE step to SE to obtain a sequence of $i$ non-empty Dyck paths, followed by a final prefix Dyck path.
  \subsubsection{Replace the prefix Dyck path by its tableau}
  Since the prefix Dyck path ends at level $2i+r-s$, the resulting tableau must have shape $(r+i,s-i)$.
\end{proof}


\section{Profiles of subspaces}
\subsection{Existence of universal formulas}
\label{sub:polynomiality} 
We first show that every partial profile is an integer partition.
\begin{lemma}
  \label{lem:ispartition}
  Suppose a linear operator $\Delta$ on $\Fq^n$ admits a subspace $W$ that has partial profile $\mu=(\mu_1,\dotsc,\mu_k)$.
  Then $\mu_1\geq \mu_2\geq \dotsb \geq \mu_k$.
\end{lemma}
\begin{proof}
  Let $W\subset\Fq^n$ have partial $\Delta$-profile $\mu$.  Let $U_0=(0)$ and $U_j=W+\Delta W\cdots+\Delta^{j-1}W$ for each $j\geq 1$. Then $\mu_j=\dim(U_j/U_{j-1})$. Since  $\Delta(U_j)\subseteq U_{j+1}$ and $U_{j+1}=U_j+\Delta U_j$, the map from the quotient space $U_j/U_{j-1}$ to $U_{j+1}/U_j$ defined by  $\alpha+U_{j-1} \mapsto \Delta\alpha+U_j$ is well-defined and surjective for each $j\geq 1$. Therefore $\mu_j\geq \mu_{j+1}$ as claimed.
  \end{proof}
  Recall the following notation from \cite{agram2022} (following \cite{sscffa}).
\begin{definition}
  Let $S_1,S_2,\ldots,S_r$ be sets of subspaces of $\Fq^n$. Let $[S_1,\ldots,S_r]_\Delta$ denote the set of all $r$-tuples $(W_1,\ldots,W_r)$ such that
  \begin{align*}
    W_i\in S_i &\quad\mbox{ for } 1\leq i\leq r,\\
    W_i\supset W_{i+1}+\Delta W_{i+1} &\quad \mbox{ for } 1\leq i\leq r-1.
  \end{align*}
\end{definition}
\begin{definition}
For nonnegative integers $a\geq b$, define
  \begin{align*}
    (a,b)_\Delta:=\{W\subset \Fq^n:\dim W=a \mbox{ and }\dim (W\cap \Delta^{-1}W)=b\},
  \end{align*}
    where $\Delta^{-1}W$ denotes the preimage of $W$ under $\Delta$.
\end{definition}
\begin{example}
  $[(8,6),(5,3)]_\Delta$ is the set of all tuples of subspaces $(W_1,W_2)$ such that $W_1\supset W_2+\Delta W_2$, and
  \begin{align*}
    \dim W_1&=8, \quad\dim (W_1\cap \Delta^{-1}W_1)=6,\\
    \dim W_2&=5, \quad \dim (W_2\cap \Delta^{-1}W_2)=3.
  \end{align*}
\end{example}
For each composition $\alpha=(\alpha_1,\ldots,\alpha_k)$ of $n$, let $X^\Delta_\alpha$ denote the number of flags $(0)=W_0\subset \dotsb \subset W_k=\Fq^n$ of $\Delta$-invariant subspaces such that  $\dim(W_j/W_{j-1})=\alpha_j$ for $1\leq j\leq k$. Given integers $a_{ij}$ $(1\leq i\leq r; 1\leq j\leq 2)$, the recurrence of Chen and Tseng \cite[Lemma 2.7]{sscffa} implies that $[(a_{11},a_{12}),\ldots,(a_{r1},a_{r2})]_\Delta$ can be written as a sum 
\begin{displaymath}
  \sum_{\alpha \models n} p_\alpha(q)X^\Delta_\alpha
\end{displaymath}
taken over all compositions $\alpha$ of $n$ for some polynomials $p_\alpha(t)\in \ZZ[t]$  independent of $q$ and $\Delta$.  
\begin{proposition}
  \label{pr:givendimensions}
  Given a weakly decreasing sequence  $\mu=(\mu_1,\ldots,\mu_k)$ of nonnegative integers, let $m_i=\mu_1+\cdots+\mu_i$ for $1\leq i\leq k$. Then 
  \begin{equation}
    \label{eq:dims}
|[(m_{k-1},m_{k-1}-\mu_k),(m_{k-2},m_{k-2}-\mu_{k-1}),\ldots, (m_{1},m_{1}-\mu_2)]_\Delta|    
  \end{equation}
is equal to the number of subspaces with partial $\Delta$-profile $\mu$.
\end{proposition}
\begin{proof}
Use the identity $\dim (W+\Delta W)=2\dim W-\dim (W\cap \Delta^{-1}W).$ By definition, the expression in \eqref{eq:dims} is equal to the number of $(k-1)$-tuples of subspaces $(W_{k-1},\ldots,W_1)$ such that
  \begin{align*}
    \dim W_i&=m_i &\mbox{ for }1\leq i\leq k-1,\\
    \dim (W_i+\Delta W_i)&=2m_i-(m_i-\mu_{i+1})=m_{i+1} &\mbox{ for }1\leq i\leq k-1,\\
    W_{i+1}&\supset W_i+\Delta W_i &\mbox{ for } 1\leq i\leq k-2.
  \end{align*}
  It follows from the dimensional constraints above that $W_{i+1}=W_i+\Delta W_i$ for $1\leq i\leq k-2$. Consequently, $W_i=W_1+\Delta W_1+\cdots +\Delta^{i-1}W_1$ for $1\leq i\leq k-1$. Moreover, if we let $W_k=W_{k-1}+\Delta W_{k-1}$, then we must have $\dim W_k=m_k$ and the proposition follows easily from these observations.
\end{proof}
\begin{corollary}
  \label{corollary:universal}
  For every integer partition $\mu$ and every composition $\alpha$ of $n$, there exists a polynomial $f_{\mu\alpha}(t)\in \ZZ[t]$ such that, for each prime power $q$ and each linear operator $\Delta$ on $\Fq^n$, the number of subspaces with partial $\Delta$-profile $\mu$ is given by
  \begin{displaymath}
    \sigma_n(\mu)=\sum_{\alpha\models n} f_{\mu\alpha}(q) X^\Delta_\alpha.
  \end{displaymath}
\end{corollary}
\begin{proof}
Follows from Proposition \ref{pr:givendimensions}. 
\end{proof}
Note that when $\Delta$ is an $n\times n$ diagonal matrix over $\Fq$ with distinct diagonal entries, $X^\Delta_\alpha$ is the multinomial coefficient ${n \choose \alpha}$, which is independent of $\Delta$.
\label{sec:profiles-subspaces}
\subsection{Counting subspaces by profile}
Let $\Delta$ be a diagonal $n\times n$ matrix over $\Fq$ with distinct entries on its diagonal.
Recall the definition of $b_\lambda(q)$ from Section \ref{sec:polyn-b_lambd}.
\begin{theorem}
  \label{theorem:counting-profile}
  For every integer partition $\mu$, the number of subspaces of $\Fq^n$ with $\Delta$-profile $\mu$ is given by
  \begin{displaymath}
    \sigma_n(\mu)=\binom n{|\mu|}(q-1)^{\sum_{j\geq 2} \mu_j}q^{\sum_{j\geq 2} \binom{\mu_j}2}b_{\mu'}(q).
  \end{displaymath}
\end{theorem}
Theorem~\ref{theorem:counting-profile} is a consequence of a refined version (Theorem \ref{theorem:refined-counting-profile}) for the number of subspaces in a given Schubert cell with a specified profile.

For each positive integer $n$, let $C(n,m)$ denote the set of all subsets of $[n]$ with cardinality $m$.
Given an element $C\in C(n,m)$ we always write $C$ as a tuple
\begin{displaymath} 
  C = (c_1,\dotsc,c_m), \text{ where } c_1<\dotsb<c_m.
\end{displaymath}
\begin{definition}[Pivots]
Every $m$-dimensional subspace $W\subset \Fq^n$ has a unique basis in reduced row echelon form, namely a basis whose elements are the rows of an $m\times n$ matrix $A=(A_{ij})$ satisfying the following conditions:
\begin{enumerate}
\item There exists a tuple $C=(c_1,\dotsc,c_m)\in C(n,m)$ (called the pivots of $W$) such that the first non-zero entry in the $i$th row of $A$ lies in the $c_i$th column and is equal to $1$.
\item $A_{i'c_i}=0$ for all $i'\neq i$ (the only non-zero entry in the $c_i$th column lies in the $i$th row).
\end{enumerate}
When this happens we say that $W$ has pivots $C$.
\end{definition}
\begin{example}
  The basis in reduced row echelon form for a subspace of $\Fq^n$ (with $n\geq 4$) with pivots $(1,3,4)$ is given by the rows of a $3\times n$ matrix of the form
  \begin{displaymath}
    \begin{pmatrix}
      1 & * & 0 & 0 & * & \dotsb & *\\
      0 & 0 & 1 & 0 & * & \dotsb & *\\
      0 & 0 & 0 & 1 & * & \dotsb & *
    \end{pmatrix},
  \end{displaymath}
  where each $*$ denotes an arbitrary element of $\Fq$.
  Moving the pivot columns to the left results in
  \begin{displaymath}
\left(    \begin{array}{ccc|cccc}
      1 & 0 & 0 & * & * & \dotsb & *\\
      0 & 1 & 0 & 0 & * & \dotsb & *\\
      0 & 0 & 1 & 0 & * & \dotsb & *
    \end{array}\right).
  \end{displaymath}
  This is a block matrix of the form $(I\mid X)$, where $I$ is the identity matrix and $X$ satisfies $X_{ij}=0$ unless $j\geq c_i-(i-1)$.
\end{example}
\begin{definition}[Shape of a matrix]
For each element $\alpha=(\alpha_1,\dotsc,\alpha_m)\in [n]^m$, a matrix $X\in M_{m\times n}(\Fq)$ is said to have shape $\alpha$ if $X_{ij}=0$ unless $j\geq \alpha_i$ for $i=1,\dotsc,m$.
\end{definition}
In this article, the shapes of interest are integer partitions.

For each tuple $C\in C(n,m)$, define $C-\delta=(c_1,c_2-1,\dotsc,c_m-(m-1))$.
In general, suppose $W\subset \Fq^n$ is a subspace with pivots $C\in C(n,m)$.
Let $A$ denote the matrix in reduced row echelon form whose row space is $W$.
Moving the pivot columns of $A$ to the left results in a matrix of the form $(I\mid X)$, where $I$ is the $m\times m$ identity matrix and $X$ is an $m\times (n-m)$ matrix with shape $C-\delta$.
It follows that the number of subspaces of $\Fq^n$ with pivots $C$ equals the number of matrices in $M_{m\times (n-m)}(\Fq)$ of shape $C-\delta$.
Thus the following well-known result \cite{MR270933} is obtained.
\begin{theorem}
  \label{theorem:q-bin-coeff}
  The number of $m$-dimensional subspaces of $\Fq^n$ with pivots $C=(c_1,\dotsc,c_m)$ is $q^{\beta_n(C)}$, where
  \begin{displaymath}
    \beta_n(C) :=|\{(c,c') \mid c<c',\,c \in C, \,c'\in [n]-C\}|= \sum_{i=1}^m (n-m-c_i+i).
  \end{displaymath}
\end{theorem}
Note that the subpaces of $\Fq^n$ with pivots $C$ are precisely the elements of the Grassmannian that lie in the Schubert cell corresponding to the partition $(n-m-c_1+1,n-m-c_2+2,\dotsc,n-c_m)$.
Theorem~\ref{theorem:q-bin-coeff} gives a combinatorial interpretation of the coefficients of ${n \brack m}_q$ which counts the number of $m$-dimensional subspaces of $\Fq^n$. The coefficient of $q^r$ in ${n \brack m}_q$ is the number of tuples $C\in C(n,m)$ for which $\beta_n(C)=r$.
\subsection{A recursion for subspaces with specified profile and pivots}
\label{sec:recursion-sigmar_nmu}
For each tuple $C\in C(n,\mu_1)$, let
\begin{displaymath}
  \phi_C:[n-\mu_1]\to [n]-C
\end{displaymath}
be the unique order-preserving bijection.
For every tuple $D\in C(n-\mu_1,\mu_2)$, let $T(C,D)$ be the tableau whose first column contains the elements of $C$, and second column contains the elements of $\phi_C(D)$ (both in increasing order).
\begin{example}
  If $n=8$, $C=(1,3,5)$ and $D=(1,3,5)$, then $\phi_C:[5]\to \{2,4,6,7,8\}$ is the unique order-preserving bijection.
  We have
  \begin{displaymath}
    T(C,D) = \ytableaushort{12,36,58},
  \end{displaymath}
  which happens to be a multilinear tableau.
  But if $n=4$, $C=(1,4)$, and $D=(1,2)$, then
  \begin{displaymath}
    T(C,D) = \ytableaushort{12,43},
  \end{displaymath}
  which is not a multilinear tableau, since its rows are not increasing.
\end{example}
\begin{definition}
  For each each integer partition $\mu=(\mu_1,\dotsc,\mu_k)$ and each $C\in C(n,\mu_1)$, let $\sigma^C_n(\mu)$ be the number of $\mu_1$-dimensional subspaces of $\Fq^n$ with pivots $C$ and profile $\mu$.
\end{definition}
\begin{theorem}
  \label{theorem:recursion}
  For each tuple $C\in C(n,\mu_1)$, and each integer partition $\mu=(\mu_1,\dotsc,\mu_k)$, we have
  \begin{displaymath}
    \sigma^C_n(\mu) = (q-1)^{\mu_2} q^{\binom{\mu_2}2}\sum_D c_q(T(C,D))\sigma^D_{n-\mu_1}((\mu_2,\dotsc,\mu_k)),
  \end{displaymath}
  the sum being over all tuples $D\in C(n-\mu_1,\mu_2)$ for which $T(C,D)$ is a multilinear tableau.
\end{theorem}
\begin{proof}
  Let $W$ be a subspace of $\Fq^n$ with pivots $C$ and $\Delta$-profile $\mu$.
  Consider the matrix $A_W$ in reduced row echelon form whose rows form a basis of $W$.
  Moving the pivots of $W$ to the left results in a block matrix of the form $(I\mid X)$, where $X\in M_{\mu_1\times(n-\mu_1)}(\Fq)$ has shape $C-\delta$.
  With respect to this permuted basis, $\Delta$ is still a diagonal matrix with distinct diagonal entries.
  Suppose that, with respect to the permuted basis, $\Delta$ has the block form
  $\begin{pmatrix}
    \Delta_1 & 0\\0&\Delta_2
  \end{pmatrix}$,
  where $\Delta_1$ is $\mu_1\times \mu_1$ and $\Delta_2$ is $(n-\mu_1)\times (n-\mu_1)$.
  The subspace $W+\Delta W+ \dotsb + \Delta^{i-1}W$ is the row space of the block matrix
  \begin{displaymath}
    \begin{pmatrix}
      I & X\\
      \Delta_1 & X\Delta_2\\
      \vdots & \vdots\\
      \Delta_1^{i-1} & X\Delta_2^{i-1}
    \end{pmatrix}.
  \end{displaymath}
  Applying the block row operation $R_r\to R_r-\Delta_1^{r-1}R_1$ for $r\geq 2$ gives 
  \begin{align*}
    \rank
    \begin{pmatrix}
      I & X\\
      \Delta_1 & X\Delta_2\\
      \vdots & \vdots\\
      \Delta_1^{i-1} & X\Delta_2^{i-1}
    \end{pmatrix} &=
    \rank
    \begin{pmatrix}
      I & X\\
      0 & X\Delta_2-\Delta_1X\\
      \vdots & \vdots\\
      0 & X\Delta_2^{i-1} - \Delta_1^{i-1}X
    \end{pmatrix} \\
    & =\mu_1 + \rank
    \begin{pmatrix}
      X\Delta_2-\Delta_1X\\
      \vdots\\
      X\Delta_2^{i-1}-\Delta_1^{i-1}X
    \end{pmatrix}.
  \end{align*}
  Let $Y=X\Delta_2-\Delta_1X$.
  Since $\Delta_1$ and $\Delta_2$ have no diagonal entries in common, $X\mapsto X\Delta_2-\Delta_1X$ is a linear automorphism of the space of $\mu_1\times(n-\mu_1)$ matrices of shape $C-\delta$. We have
  \begin{displaymath}
    X\Delta_2^r - \Delta_1^r X = Y\Delta_2^{r-1} + \Delta_1Y\Delta_2^{r-2} + \dotsb + \Delta_1^{r-2}Y\Delta_2 + \Delta_1^{r-1}Y.
  \end{displaymath}
  Therefore, $\dim(W+ \Delta W \dotsb + \Delta^{i-1}W)-\mu_1$ is the rank of the matrix
  \begin{displaymath}
    \begin{pmatrix}
      Y\\
      Y\Delta_2+\Delta_1Y\\
      \vdots\\
      Y\Delta_2^{i-2} + \Delta_1Y\Delta_2^{i-3} + \dotsb + \Delta_1^{i-3}Y\Delta_2 + \Delta_1^{i-2}Y
    \end{pmatrix}.
  \end{displaymath}
  Applying the block row operation $R_r\to R_r- \Delta_1R_{r-1}$ in the order $r=i-1,i-2,\ldots,2$ to the matrix above gives
  \begin{displaymath}
    \begin{pmatrix}
      Y\\
      Y\Delta_2\\
      \vdots\\
      Y\Delta_2^{i-2}
    \end{pmatrix}.
  \end{displaymath}
  Therefore the row space $\widetilde W$ of $Y$ has $\Delta_2$-profile $(\mu_2,\dotsc,\mu_k)$. We have thus established a map from the set of subspaces $W \subset \Fq^n$  with pivots  $C$  and $\Delta$-profile $\mu$ to the set of subspaces $\widetilde{W} \subset \Fq^{n-\mu_1}$  with $\Delta_2$-profile $\tilde{\mu}$ where $\tilde{\mu} := (\mu_2, \ldots, \mu_k)$. The fiber of such a \(\widetilde{W}\) is in bijection with \(\{Y \in M_{\mu_1 \times (n-\mu_1)}(\Fq) \mid Y \text{ has shape } C-\delta \text{ and } Y\text{ has row space }\widetilde{W}\}\).

  For each tuple $D=(d_1,\ldots,d_{\mu_2})\in C(n-\mu_1,\mu_2)$, let $\eta(C,D)$ be the number of $\mu_1\times(n-\mu_1)$ matrices $Y$ of shape $C-\delta$ whose row space is a fixed subspace $\widetilde W\subset \Fq^{n-\mu_1}$ with pivots $D$. We shall see that this quantity depends only on $D$ and not on the choice of $\widetilde W$, in which case we have
  \begin{displaymath}
    \sigma_n^C(\mu) = \sum_{D\in C(n-\mu_1,\mu_2)} \eta(C,D) \sigma_{n-\mu_1}^D((\mu_2,\dotsc,\mu_k)).
  \end{displaymath}
  It remains to show that for every subspace $\widetilde W\subset \Fq^{n-\mu_1}$ with pivots $D$, 
  \begin{equation}
    \label{eq:eta}
    \eta(C,D) = (q-1)^{\mu_2}q^{\binom{\mu_2}2}c_q(T(C,D))
  \end{equation}
  whenever $T(C,D)$ is a multilinear tableau, and $0$ otherwise.
  
  In order to prove \eqref{eq:eta}, we need to count the number of matrices $Y$ of shape $C-\delta$ whose row space is $\widetilde W$.
  We will do so by expanding the rows of $Y$ in terms of the rows of the $\mu_2\times (n-\mu_1)$ matrix $E$ in reduced row echelon form whose rows span $\widetilde W$.
  Let $E_j$ denote the $j$th row of $E$ for each $1\leq j\leq \mu_2$.
  For each $i\in [\mu_1]$, write the $i$th row of $Y$ as 
  \begin{displaymath}
    Y_i = \sum_{j=1}^{\mu_2} a_{ij}E_j,\text{ for } i=1,\dotsc,\mu_1.
  \end{displaymath}
  The  $\mu_1\times \mu_2$ matrix $A=(a_{ij})$ has rank $\mu_2$.
  
  Since $E$ has pivots $D$, the first non-zero entry of $E_j$ lies in the $d_j$th column.
  Therefore, since $Y$ has shape $C-\delta$, if $a_{ij}\neq 0$ then $d_j\geq c_i-(i-1)$.
  Thus, for each $j$, the number of potentially non-zero elements in the $j$th column of $A$ is
  \begin{displaymath}
    \#\{i\mid d_j\geq c_i-(i-1)\} = c_{j2}(T(C,D)) + (j-1),
  \end{displaymath}
  since $d_j\geq c_i-(i-1)$ if and only if $\phi_C(d_j)>c_i$.
  In particular, the first column of $A$ is any non-zero vector in $\Fq^{c_{12}(T(C,D))}$, for which there are $q^{c_{12}(T(C,D))}-1$ possibilities.
  The second column of $A$ is any non-zero vector in $\Fq^{c_{22}(T(C,D))+1}$ that is independent of the first column.
  Thus there are $q^{c_{22}(T(C,D))+1}-q$ possibilities for the second column of $A$.
  In general, having chosen the first $j-1$ columns of $Y$, the $j$th column is a vector in $\Fq^{c_{j2}(T(C,D))+(j-1)}$ that is independent of the first $j-1$ columns of $A$.
  Thus there are $q^{c_{j2}(T(C,D))+(j-1)}-q^{j-1}$ possibilities for the the $j$th column of $A$.
  Therefore
  \begin{displaymath}
    \eta(C,D) = \prod_{j=1}^{\mu_2} (q^{c_{j2}(T(C,D))+(j-1)}-q^{j-1}) = (q-1)^{\mu_2}q^{\binom{\mu_2}2}c_q(T(C,D)),
  \end{displaymath}
  proving \eqref{eq:eta}.
\end{proof}

Recall that $\Tab_{\subset [n]}(\lambda)$ is the set of multilinear tableaux of shape $\lambda$ whose support is a subset of $[n]$.
Theorem~\ref{theorem:recursion} has the following consequence:
\begin{theorem}
  \label{theorem:refined-counting-profile}
  For each partition $\mu=(\mu_1,\dotsc,\mu_k)$ and each $C\in C(n,\mu_1)$, we have
  \begin{displaymath}
    \sigma^C_n(\mu) = (q-1)^{\sum_{j\geq 2}\mu_j}q^{\sum_{j\geq 2}\binom{\mu_j}2} \sum_{T\in \Tab_{\subset [n]}(\mu') \text{ has first column }C} c_q(T).
  \end{displaymath}
\end{theorem}
\begin{proof}
  Induct on $k$, the number of parts of $\mu$.
  If $k=1$, then $\sigma_n^C(\mu)=1$, and there is nothing to prove.
  Otherwise, let $\tilde\mu=(\mu_2,\dotsc,\mu_k)$.
  By the induction hypothesis, we have
  \begin{displaymath}
    \sigma^D_{n-\mu_2}(\tilde\mu) =(q-1)^{\sum_{i\geq 3}\mu_i}q^{\sum_{i\geq 3}\binom{\mu_i}2}\sum_{\tilde T\in \Tab_{\subset[n-\mu_1]}(\tilde\mu') \text{ has first column }D}c_q(\tilde T)
  \end{displaymath}
  for every tuple $D\in C(n-\mu_1,\mu_2)$. The bijection $\phi_C:[n-\mu_1]\to [n]-C$ induces a bijection
  \begin{displaymath}
    \Tab_{\subset [n-\mu_1]}(\tilde\mu)\to \Tab_{\subset ([n]-C)}(\tilde\mu)
  \end{displaymath}
  which preserves $c_q$.
  Therefore,
  \begin{displaymath}
    \sigma^D_{n-\mu_2}(\tilde\mu) =(q-1)^{\sum_{i\geq 3}\mu_i}q^{\sum_{i\geq 3}\binom{\mu_i}2}\sum_{\parbox{3cm}{\tiny{$\title T\in \Tab_{\subset([n]-C)}(\tilde\mu')$\\has first column $\phi_C(D)$}}}c_q(\tilde T).
  \end{displaymath}
  Hence the recursive identity of Theorem~\ref{theorem:recursion} can be written as
  \begin{displaymath}
    \sigma^C_n(\mu) =
    (q-1)^{\mu_2}q^{\binom{\mu_2}2}\sum_{D,\tilde T}c_q(T(C,D)) (q-1)^{\sum_{i\geq 3}\mu_i}q^{\sum_{i\geq 3}\binom{\mu_i}2}c_q(\tilde T),
  \end{displaymath}
  the sum being over all pairs $(D,\tilde T)$ where $D\in C(n-\mu_1,\mu_2)$ and $\tilde T\in \Tab_{\subset([n]-C)}(\tilde\mu')$ with first column $\phi_C(D)$ are such that $T(C,D)$ is a multilinear tableau.
  Such pairs $(D,\tilde T)$ are in bijection with tableaux $T\in \Tab_{\subset[n]}(\mu')$ with first column $C$; the tableau $T$ associated to $(D,\tilde T)$ has first column $C$ and the remaining columns are those of $\tilde T$.
  Therefore
  \begin{displaymath}
    \sigma^C_n(\mu) = (q-1)^{\sum_{j\geq 2}\mu_j}q^{\sum_{j\geq 2}\binom{\mu_j}2}\sum_{T\in \Tab_{\subset[n]}(\mu') \text{ has first column } C}c_q(T),
  \end{displaymath}
  as claimed.
\end{proof}
Taking the sum over all tuples $C\in C(n,\mu_1)$ in Theorem~\ref{theorem:refined-counting-profile} gives Theorem~\ref{theorem:counting-profile}.

The decomposition of a Schubert cell with respect to profiles leads to the following combinatorial identity:
\begin{corollary}
  For each tuple $C\in C(n,m)$,
  \begin{displaymath}
    q^{\beta_n(C)} = \sum_{\{\mu\mid \mu_1=m\}} (q-1)^{\sum_{j\geq 2} \mu_j}q^{\sum_{j\geq 2}\binom{\mu_j}2}\sum_{T\in \Tab_{\subset [n]}(\mu') \text{ has first column }C} c_q(T).
  \end{displaymath}
\end{corollary}
\begin{example}
  Let $n=4$ and $C=(1,3)$.
  The five multilinear tableaux with support contained in $[4]$ with first column $(1,3)$ and their contributions to $q^{\beta_4(C)}=q^3$ are as follows:
  \ytableausetup{smalltableaux}
  \begin{displaymath}
    \begin{array}{{c|ccccc}}
      \hline
      & & & & & \\
      \text{Tableau} & \ytableaushort{1,3} & \ytableaushort{12,3} & \ytableaushort{14,3} & \ytableaushort{12,34}&\ytableaushort{124,3}\\
      & & & & &\\
      \text{Contribution} & 1 & (q-1) & (q-1)(q+1) & (q-1)^2q & (q-1)^2\\
      & & & & &\\
      \hline
    \end{array}
  \end{displaymath}
\end{example}
\subsection{Splitting subspaces}
\label{sec:chord} 
When $n=md$, Theorem \ref{theorem:counting-profile} gives the number of $\Delta$-splitting subspaces (see Eq. \eqref{eq:defsplit}) as
$$
\sigma_n(m^d)=(q-1)^{m(d-1)} q^{{m \choose 2}(d-1)}b_{(d^m)}(q).
$$
By Eq. \eqref{eq:blamset} we have
$$
b_{(d^m)}(q)=\sum_{\AA \in \Pi_n(d^m)}q^{v(\AA)}.
$$
When $d=2$, this is Touchard's polynomial $T_m(q)$ (see Section~\ref{sec:relation-q-hermite}).
\subsection{Partial profiles}
\label{sec:partialprofiles}
For every pair $C, C'\subset \nat$, define
\begin{displaymath}
  \beta(C,C') = |\{(c,c')\mid c<c',\,c\in C,c'\in C'\}|. 
\end{displaymath}
\begin{definition}
  For every positive integer $n$ and tableau $T\in \Tab_{\subset[n]}$, let $C_\text{last}$ denote the last column of $T$. Let
  \begin{displaymath}
    \gamma_n(T) = \beta(C_{\text{last}},[n]-\supp(T)).
  \end{displaymath}
\end{definition}
\begin{theorem}
  \label{theorem:partial_profiles-pivots}
  Let $\mu=(\mu_1,\dotsc,\mu_k)$ be a partition and suppose $n$ is a postive integer.
  For every tuple $C\in C(n,\mu_1)$, the number of subspaces of $\Fq^n$ with pivots $C$ and partial profile $\mu$ is given by
  \begin{displaymath}
    \pi^C_n(\mu) = (q-1)^{\sum_{j\geq 2}\mu_j}q^{\sum_{j\geq 2}\binom{\mu_j}2}\sum_{T\in \Tab_{\subset[n]}(\mu')\text{ has first column }C} c_q(T)q^{\gamma_n(T)}.
  \end{displaymath}
\end{theorem}
Summing over all tuples $C\in C(n,\mu_1)$ gives the following result.
\begin{theorem}
  \label{theorem:partial_profiles}
  Let $\mu=(\mu_1,\dotsc,\mu_k)$ be a partition and $n\in \nat$.
  The number of subspaces of $\Fq^n$ with partial profile $\mu$ is given by
  \begin{displaymath}
    \pi_n(\mu) = (q-1)^{\sum_{j\geq 2}\mu_j}q^{\sum_{j\geq 2}\binom{\mu_j}2}\sum_{T\in \Tab_{\subset[n]}(\mu')} c_q(T)q^{\gamma_n(T)}.
  \end{displaymath}
\end{theorem}
The proof of Theorem~\ref{theorem:partial_profiles-pivots} is essentially the same as that of Theorem~\ref{theorem:refined-counting-profile}.
Only the base case is different.
When $\mu$ has only one part, say $\mu=(m)$, $\pi_n(C)$ is given by Theorem~\ref{theorem:q-bin-coeff}.
Since the remaining steps in the proof are the same, they are omitted.

\begin{corollary}
  \label{cor:numantiinv}
  The number of $l$-fold anti-invariant subspaces of dimension $m$ of a regular diagonal matrix on $\Fq^n$ equals
  $$
(q-1)^{lm}q^{l{m \choose 2}}\sum_{T\in \Tab_{\subset[n]}((l+1)^m)} c_q(T)q^{\gamma_n(T)}.   
  $$
\end{corollary}
\begin{proof}
Since anti-invariant subspaces are those with partial profile $\mu=(m^{l+1})$ (see Eq. \eqref{eq:defanti}) the result follows from Theorem \ref{theorem:partial_profiles}. 
\end{proof}
\subsection{Combinatorial interpretation of the $q$-Stirling numbers}
Given a subspace $W\subset \Fq^n$, let $r(W)$ denote the dimension of the smallest $\Delta$-invariant subspace containing $W$.
Thus
\begin{displaymath}
  r(W) = \dim \sum_{j\geq 0} \Delta^j W.
\end{displaymath}
\begin{lemma}
  \label{lemma:combin-q-stir}
  Let $m$, $n$, and $r$ be non-negative integers, with $n\geq r\geq m$. 
  Then the number of subspaces $W\subset \Fq^n$ of dimension $m$ such that $r(W)=r$ is given by
  \begin{equation}
    \label{eq:Carlitz-r}
    (q-1)^{r-m}\binom nr S_q(r,m).
  \end{equation}
\end{lemma}
\begin{proof}
  By Theorem \ref{theorem:counting-profile}, the number of subspaces $W\subset \Fq^n$ of dimension $m$ such that $r(W)=r$ is given by
  \begin{align*}
    \sum_{\mu\vdash r,\;\mu_1=m} \sigma_n(\mu) & = \sum_{\mu\vdash r, \;\mu_1=m} \binom nr(q-1)^{\sum_{j\geq 2}\mu_j}q^{\sum_{j\geq 2}\binom{\mu_2}2}b_{\mu'}(q)\\
    & = (q-1)^{r-m}\binom nr S_q(r,m),
  \end{align*}
  using the identity \ref{eq:q-stirling-alt}.
\end{proof}
\subsection{Bijective proof of Carlitz's identity}
Recall that the $q$-binomial coefficient ${n \brack m}_q$ counts the number of subspaces $W$ of $\Fq^n$ of dimension $m$.
Each such subspace must have $m\leq r(W)\leq n$.
Summing over all possibilities, we get a bijective proof of Carlitz's identity \cite[Eq. (8)]{MR1501675} (see also \cite[Eq. (3.4)]{MR27288}):
\begin{equation}
  \label{eqn:carlitz}
  {n \brack m}_q = \sum_{r=m}^n (q-1)^{r-m}\binom nr S_q(r,m).
\end{equation}
To the best of our knowledge, this is the first bijective proof of Carlitz's identity.
 
 
\subsection*{Acknowledgements}
We thank the anonymous referees for their detailed and helpful comments.
We used SageMath \cite{sagemath} extensively as an experimental tool to formulate many of our results.
\printbibliography
\end{document}